\documentclass[11pt]{amsart}

\usepackage{amsmath}
\usepackage{amssymb}
\usepackage{pdfsync}

\usepackage{tikz}

\newcommand{\bS}{{\mathbb S}}  
\newcommand{\R}{{\mathbb R}}       
\newcommand{\HH}{{\mathcal H}}

\newcommand{\PP}{{\mathcal P}}

\newcommand\eqn[1]{\eqref{e:#1}} 
\def\bB{\mathbb{B}}
\def\d{\partial}
\def\ext{\textrm{ext}}
\def\bB{\mathbb{B}}

\newcommand{\RR}{{\mathcal R}}
\newcommand{\MM}{{\mathcal M}}
\newcommand{\NN}{{\mathcal N}}

\newcommand{\EE}{{\mathcal E}}

\newcommand{\diam}{\mathop{\rm diam}}
\newcommand{\dist}{{\rm dist}}
\newcommand{\ds}{\displaystyle }

\newcommand{\fiproof}{{\hspace*{\fill} $\square$ \vspace{2pt}}}

\newcommand{\rf}[1]{{(\ref{#1})}}

\newcommand{\supp}{\operatorname{supp}}

\newcommand{\vphi}{{\varphi}}
\newcommand{\ve}{{\varepsilon}}
\newcommand{\vv}{{\vspace{2mm}}}
\newcommand{\vvv}{\vspace{4mm}}
\newcommand{\wt}[1]{{\widetilde{#1}}}

\newcommand{\noi}{\noindent}

\newcommand{\dv}{\mathop\mathrm{div }}			

\def\Xint#1{\mathchoice
{\XXint\displaystyle\textstyle{#1}}%
{\XXint\textstyle\scriptstyle{#1}}%
{\XXint\scriptstyle\scriptscriptstyle{#1}}%
{\XXint\scriptscriptstyle\scriptscriptstyle{#1}}%
\!\int}
\def\XXint#1#2#3{{\setbox0=\hbox{$#1{#2#3}{\int}$ }
\vcenter{\hbox{$#2#3$ }}\kern-.58\wd0}}

\def\avint{\Xint-}

\newtheorem{theorem}{Theorem}[section]
\newtheorem{lemma}[theorem]{Lemma}

\newtheorem{coro}[theorem]{Corollary}
\newtheorem{propo}[theorem]{Proposition}

\newtheorem*{lemma*}{Lemma}
\newtheorem*{theorema}{Theorem A}
\newtheorem*{theoremb}{Theorem B}

\theoremstyle{definition}
\newtheorem{definition}[theorem]{Definition}

\theoremstyle{remark}
\newtheorem{rem}[theorem]{\bf Remark}

\numberwithin{equation}{section}

\newcommand{\brem}{\begin{rem}}
\newcommand{\erem}{\end{rem}}

\newcommand\Lemma[1]{Lemma \ref{l:#1}}

\def\loc{{\textrm loc}}
\def\VMO{{\textrm{VMO}}}

\begin{document}

\title[The one-phase problem for harmonic measure]{The one-phase problem for harmonic measure in two-sided NTA domains}


\author[Azzam]{Jonas Azzam}
\address{Jonas Azzam
\\
Departament de Matem\`atiques
\\
Universitat Aut\`onoma de Barcelona
\\
Edifici C Facultat de Ci\`encies
\\
08193 Bellaterra (Barcelona), Catalonia
}
\email{jazzam@mat.uab.cat}

\author[Mourgoglou]{Mihalis Mourgoglou}

\address{Mihalis Mourgoglou
\\
Departament de Matem\`atiques\\
 Universitat Aut\`onoma de Barcelona
\\
Edifici C Facultat de Ci\`encies
\\
08193 Bellaterra (Barcelona), Catalonia
}
\email{mourgoglou@mat.uab.cat}

\author[Tolsa]{Xavier Tolsa}
\address{Xavier Tolsa\\
ICREA,  Universitat Aut\`onoma de Barcelona, and BGSMath \\  Departament de Matem\`atiques
\\
Edifici C Facultat de Ci\`encies
\\
08193 Bellaterra (Barcelona), Catalonia}
\email{xtolsa@mat.uab.cat}

\thanks{The three authors were supported by the ERC grant 320501 of the European Research Council (FP7/2007-2013). X.T. was also supported by 2014-SGR-75 (Catalonia), MTM2013-44304-P (Spain), and by the Marie Curie ITN MAnET (FP7-607647).}

\begin{abstract}
We show that if $\Omega\subset\R^3$ is a two-sided NTA domain with AD-regular boundary such that the logarithm of the Poisson kernel belongs to $\textrm{VMO}(\sigma)$, where $\sigma$ is the surface measure of $\Omega$, then the outer unit normal to $\partial\Omega$ belongs to $\textrm{VMO}(\sigma)$ too.
The analogous result fails for dimensions larger than $3$. This answers a question posed by Kenig and Toro and also by Bortz and Hofmann.
\end{abstract}

\maketitle

\section{Introduction}

In this paper we study a one-phase free boundary problem in connection with the Poisson kernel. The study of this type of problems was initiated by Alt and Caffarelli in their pioneering work \cite{Alt-Caffarelli}, where they showed that for a Reifenberg flat domain 
with $n$-AD-regular boundary in $\R^{n+1}$, if the logarithm of the Poisson kernel is in $C^\alpha$ for some $\alpha>0$, then the domain is of class $C^{1,\beta}$ for some $\beta>0$.
Later on, Jerison \cite{Jerison} showed that, in fact, one can take $\beta=\alpha$. In the works \cite{KT97}, \cite{Kenig-Toro-annals} and \cite{Kenig-Toro2},
Kenig and Toro considered the endpoint case of the logarithm of the Poisson kernel being in VMO, and they obtained the following remarkable result:

\begin{theorema}[\cite{Kenig-Toro2}]\label{thm:KTannal}
Suppose $\Omega\subset \mathbb{R}^{n+1}$ is a $\delta$-Reifenberg flat chord-arc domain for some $\delta>0$ small enough. Denote by $\sigma$ the surface measure of $\Omega$ and by
$h$ the Poisson kernel with a pole in $\Omega$ if $\Omega$ is bounded or with the pole at infinity if $\Omega$ is unbounded. Then 
$\log h\in \textrm{VMO}(\sigma)$ if and only if 
the outer unit normal $\vec n$ to $\partial\Omega$ is in $\textrm{VMO}(\sigma)$. 
\end{theorema}

A domain $\Omega\subset \mathbb{R}^{n+1}$ is called {\it chord-arc} if it is an NTA domain with $n$-AD-regular boundary. Its Poisson kernel with pole at $p\in\Omega$ equals $h=\frac{d\omega^p}{d\sigma}$, where $\omega^p$ stands for the harmonic measure of $\Omega$ with pole at $p$. 
For the definitions of Reifenberg flatness, NTA, and VMO, we refer the reader to the preliminaries section. 

We also remark that, in fact, Kenig and Toro proved in \cite{Kenig-Toro2} a slightly weaker statement that the one in Theorem \ref{thm:KTannal}. Indeed, instead of
showing that when $\log h\in \textrm{VMO}(\sigma) $ then
the outer unit normal $\vec n$ to $\partial\Omega$ is in $\textrm{VMO}(\sigma)$, they proved that $\vec n$ belongs to $\textrm{VMO}_{loc}(\sigma)$ (which coincides with $\textrm{VMO}(\sigma)$ when $\Omega$ is bounded). However, as we explain in Remark \ref{rem***}, a minor modification of their arguments in \cite{Kenig-Toro2} proves the full statement above in Theorem \ref{thm:KTannal}.

Without the Reifenberg flatness assumption and just assuming the NTA condition, the conclusion of the theorem above need not hold: in \cite[Proposition 3.1]{Kenig-Toro-annals}, Kenig and Toro showed that for the Kowalski-Preiss cone $\Omega=\{(x,y,z,w):x^{2}+y^{2}+z^{2}>w^{2}\}\subset \mathbb{R}^{4}$, the harmonic measure with pole at infinity coincides with the surface measure modulo a constant factor, and thus 
has $\log h\in \textrm{VMO}(\sigma)$, even though the outer unit normal is not in $\textrm{VMO}(\sigma)$. In fact, a similar conical example in $\R^3$ was shown previously by Alt and Caffarelli in \cite[Section 2.7]{Alt-Caffarelli}.

It was conjectured by Kenig and Toro \cite{Kenig-Toro} and Bortz and Hofmann
\cite{BH} that, instead of the Reifenberg flatness assumption,  being a two-sided chord-arc domain should be enough for the implication $\log h\in \textrm{VMO}(\sigma) \Rightarrow 
\vec n\in \textrm{VMO}(\sigma)$.  
By a two-sided chord-arc domain we mean a chord-arc domain such that its exterior is also connected and chord-arc. Kenig and Toro showed that this holds under the additional assumption that the logarithm of the Poisson kernel of  the exterior domain is also in $\textrm{VMO}(\sigma)$. Their precise result reads as follows:

\begin{theoremb}[{\cite[Corollary 5.2]{Kenig-Toro}}]\label{t:kt2} Let $\Omega$ be a two-sided chord-arc domain in $\mathbb{R}^{n+1}$. Assume further that $\log \frac{d\omega}{d\sigma},\log \frac{d\omega_{\ext}}{d\sigma}\in \textrm{VMO}_{\loc}(\sigma)$. Then $\vec{n}\in \textrm{VMO}_{\loc}(\sigma)$.
\end{theoremb}

In \cite{BH} Bortz and Hofmann showed that this same result holds under the assumption that  $\d\Omega$ is uniformly $n$-rectifiable and so that the measure theoretic boundary has full surface measure, instead  of the two-sided chord-arc condition above. The boundary of any
two-sided chord-arc domain is always uniformly $n$-rectifiable by results due to David and Jerison \cite{DJ}, and thus this is more general than Theorem B. We also note that,
by Proposition 4.10 in \cite{HMT}, such domains with $\vec{n}\in \VMO_{\loc}(\sigma)$ are also vanishing Reifenberg flat. It is also worth mentioning that the arguments in \cite{BH} are very different from the ones in \cite{Kenig-Toro}: while in \cite{Kenig-Toro} Kenig and Toro use blow-up techniques, in \cite{BH} 
Bortz and Hofmann rely on the relationship between the Riesz transform and harmonic measure and 
and exploit the jump relations for the gradient of the single layer potential.  

%

In this paper we resolve the conjecture mentioned above:

\begin{theorem}\label{teo}
Let $\Omega\subset \R^3$ be a two-sided chord-arc domain. Denote by $\sigma$ the surface measure of $\Omega$ and by
$h$ the Poisson kernel with a pole in $\Omega$ if $\Omega$ is bounded or with the pole at infinity if $\Omega$ is unbounded. 
If $\log h\in \textrm{VMO}(\sigma)$, then the outer unit normal of $\Omega$ also belongs to $\textrm{VMO}(\sigma)$.

On the other hand, for $d\geq4$, there are two-sided chord-arc domains $\Omega\subset\R^d$ satisfying
$h \equiv1$ and such that the outer unit normal of $\Omega$ does not belong to $\textrm{VMO}(\sigma)$.
\end{theorem}

Most of the paper is devoted to prove the positive result stated in the theorem for $\R^3$. 
Our arguments use the powerful blow-up techniques developed by Kenig and Toro in \cite{Kenig-Toro2}. Indeed, by arguments analogous to the ones of Kenig and Toro, we reduce our problem to the study of the case
when $\Omega_\infty$ is an unbounded two-sided chord-arc domain such that its Poisson kernel with pole at infinity is constantly equal to $1$. By combining a monotonicity formula due to Weiss \cite{Weiss} and some topological arguments inspired by a work from Caffarelli, Jerison and Kenig \cite{CJK}, we then show that for such domains all blow-downs are flat. This is probably one
of the main novelties in our paper. Then an application of a variant of a well known theorem of Alt and Caffarelli \cite{Alt-Caffarelli} shows that $\Omega_\infty$ must be a half space.

The aforementioned reduction of the problem to the case 
when the Poisson kernel with pole at infinity is constantly equal to $1$ requires estimating from above the gradient of the Green function. This estimate is obtained in \cite{Kenig-Toro2} under the assumption that the domain $\Omega$ is Reifenberg flat, and this is one of the main technical difficulties of that paper. In \cite{Kenig-Toro} it is shown how these estimates can be extended to the case when $\Omega$ is not Reifenberg flat. In our present paper we provide some alternative arguments to estimate the gradient of the Green function. The main difference with respect to the ones in \cite{Kenig-Toro2} and
\cite{Kenig-Toro} is that in the present paper we use the jump relations for the gradient of single layer potentials, instead of the (perhaps) less standard approach in the aforementioned works. We think that our approach has some independent interest (specially because of the connection between harmonic measure with pole at infinity and the Riesz transform that we describe in Section \ref{sec2}).

Concerning the negative result for dimensions $d\geq4$ in Theorem \ref{teo}, basically we recall in the last section of the paper an example of a conical domain in $\R^4$ by Guanghao Hong\footnote{ So the statement
in the theorem referring to the case $d\geq4$ should not be attributed to us. 
} \cite{Hong} such that the harmonic measure with pole at infinity coincides with
surface measure, and so that the outer unit normal does not belong to $\textrm{VMO}(\sigma)$. 
 One can check easily that such domain is two-sided
NTA. Probably, this example was unnoticed in some recent works in this area.

\vspace{3mm}
We would like to thank Simon Bortz and Tatiana Toro for explaining to us some of the result in \cite{Kenig-Toro2}, of which we were not aware.


\vv

\section{Preliminaries}

For $a,b\geq 0$, we will write $a\lesssim b$ if there is $C>0$ so that $a\leq Cb$ and $a\lesssim_{t} b$ if the constant $C$ depends on the parameter $t$. We write $a\approx b$ to mean $a\lesssim b\lesssim a$ and define $a\approx_{t}b$ similarly.

\begin{definition}
Given a closed set $E$, $x\in \mathbb{R}^{d}$, $r>0$, and $P$ a $d$-plane, we set
\[\Theta_{E}(x,r,P)=r^{-1}\max\left\{\sup_{y\in E\cap B(x,r)}\dist(y,P), \sup_{y\in P\cap B(x,r)}\dist(y,E)\right\}.\]
Also define
\[
\Theta_{E}(x,r)=\inf_{P}\Theta_{E}(x,r,P)\]
where the infimum is over all $d$-planes $P$. A set $E$ is {\it $\delta$-Reifenberg flat} if $\Theta_{E}(x,r)<\delta$ for all $x\in E$ and $r>0$, and is {\it vanishing Reifenberg flat} if 
\[
\lim_{r\rightarrow 0} \sup_{x\in E} \Theta_{E}(x,r)=0.\]
 \end{definition}

\begin{definition}\label{defreif}
Let $\Omega\subset \R^{n+1}$ be an open set, and let $0<\delta<1/2$. We say that $\Omega$
is a $\delta$-Reifenberg flat domain if it satisfies the following conditions:
\begin{itemize}
\item[(a)] $\partial\Omega$ is $\delta$-Reifenberg flat.

\item[(b)] For every $x\in\partial \Omega$ and $r>0$, denote by $\PP(x,r)$ an $n$-plane that minimizes $\Theta_{E}(x,r)$. Then one of the connected components of 
$$B(x,r)\cap \bigl\{x\in\R^{n+1}:\dist(x,\PP(x,r))\geq 2\delta\,r\bigr\}$$
is contained in $\Omega$ and the other is contained in $\R^{n+1}\setminus\Omega$.
\end{itemize}
If, additionally,  $\partial\Omega$ is vanishing Reifenberg flat, then $\Omega$ is said to be vanishing Reifenberg flat, too.
\end{definition}

\begin{definition}
 Let $\Omega\subset \mathbb{R}^{n+1}$.
We say that $\Omega$ satisfies the {\it Harnack chain condition} if there is a uniform constant $C$ such that
for every $\rho >0,\, \Lambda\geq 1$, and every pair of points
$x,y \in \Omega$ with $\dist(x,\d\Omega),\,\dist(y,\d\Omega) \geq\rho$ and $|x-y|<\Lambda\,\rho$, there is a chain of
open balls
$B_1,\dots,B_N \subset \Omega$, $N\leq C(\Lambda)$,
with $x\in B_1,\, y\in B_N,$ $B_k\cap B_{k+1}\neq \emptyset$
and $C^{-1}\diam (B_k) \leq \dist (B_k,\partial\Omega)\leq C\diam (B_k).$  The chain of balls is called
a {\it Harnack Chain}. Note that if such a chain exists, then
\[u(x)\approx_{N}u(y).\]
For $C\geq 2$, $\Omega$ is a {\it $C$-corkscrew domain} if for all $\xi\in \d\Omega$ and $r>0$ there are two balls of radius $r/C$ contained in $B(\xi,r)\cap \Omega$ and $B(\xi,r)\backslash \Omega$ respectively. If $B(x,r/C)\subseteq B(\xi,r)\cap \Omega$, we call $x$ a {\it corkscrew point} for the ball $B(\xi,r)$. Finally, we say $\Omega$ is {\it $C$-nontangentially accessible (or $C$-NTA, or just NTA)} if it satisfies the Harnack chain condition and  is a $C$-corkscrew domain. We say $\Omega$ is {\it two-sided $C$-NTA} if both $\Omega$ and $\Omega_{\ext}:=(\overline{\Omega})^{c}$ are $C$-NTA. Finally, it is {\it chord-arc} if,
additionally, $\d\Omega$ is $n$-AD-regular, meaning there is $C>0$ so that, if $\sigma$ denote surface measure, then
\[
C^{-1} r^{n} < \sigma(B(x,r))< Cr^{n} \mbox{ for all }x\in \d\Omega, \;\; 0<r\leq\diam(\Omega).\]
\end{definition}
 
 Any measure $\sigma$ that satisfies the preceding estimate for all $x\in \supp\sigma$ and $0<r\leq\diam(\supp\sigma)$ is called
 $n$-AD-regular.

\begin{definition}\label{defvm00}
Let $\sigma$ be an n-AD-regular measure in $\mathbb{R}^{n}$ and $f$ a locally integrable function with respect to $\sigma$. We say $f\in \textrm{VMO}(\sigma)$ if
\begin{equation}\label{defvmo}
\lim_{r\rightarrow 0}\sup_{x\in \supp \sigma}\,\, \avint_{B(x,r)} \left| f-\,\avint_{B(x,r)}f\,d\sigma\right|^{2}d\sigma =0.
\end{equation}
We say $f\in \textrm{VMO}_{\textrm{loc}}(\sigma)$ if, for any compact set $K$,
\[
\lim_{r\rightarrow 0}\sup_{x\in \supp \sigma\cap K} \,\,\avint_{B(x,r)} \left|f-\,\avint_{B(x,r)}f\,d\sigma\right|^{2}d\sigma =0.\]
\end{definition}

It is well known  that the space VMO  coincides with the closure of the set of bounded uniformly continuous functions on $\supp \sigma$ in the BMO norm. 

We also remark that one can find slightly different definitions of VMO in the literature. 
For example, in some references besides \rf{defvmo} one asks the additional condition that
$$\lim_{r\rightarrow \infty}\sup_{x\in \supp \sigma} \avint_{B(x,r)} \left| f-\avint_{B(x,r)}fd\sigma\right|^{2}d\sigma =0.$$
In this case, it turns out that VMO coincides with the closure of the set of {\em compactly supported} continuous functions on $\supp \sigma$ in the BMO norm. However, the definition we will use in our paper is Definition \ref{defvm00} (as in other works like \cite{Kenig-Toro-annals}
or \cite{Kenig-Toro2}) .

\vv

\section{The Riesz transform of the harmonic measure with pole at infinity}\label{sec2}

We warn readers that are familiar with the results in \cite{Kenig-Toro2} and \cite{Kenig-Toro} that they may skip 
this section, as well as Sections \ref{section4} and \ref{section5}, and go directly to Section \ref{section6} without much harm.
In fact, in Sections 3-5 we provide the alternative arguments to estimate the gradient of the Green function that we mentioned
in the Introduction. Our approach uses the jump relations for the gradient of the single layer potential (derived by 
Hofmann, Mitrea, and Taylor \cite{HMT} in the case of chord-arc domains and somewhat more general settings). Modulo
these standard relations our arguments are reasonably self contained and shorter than the ones in \cite{Kenig-Toro2} and \cite{Kenig-Toro}.

Recall from \cite[Lemma 3.7]{Kenig-Toro-annals} that if $\Omega\subset\R^{n+1}$ is an unbounded NTA domain, then there exists a function $u\in C(\overline\Omega)$
and a measure $\omega$ in $\partial \Omega$
such that
\begin{equation}\label{eq:hminfty}
\left\{ \begin{array}{ll} \Delta u = 0 & \mbox{ in $\Omega$,}\\
u=0  & \mbox{ in $\partial \Omega$,}\\
u > 0 & \mbox{ in $\Omega$,}
\end{array}
\right.
\end{equation}
and
\begin{equation}\label{eq:Green-hminfty}
\int_\Omega u \,\Delta\vphi\,dm = \int_{\partial\Omega} \vphi\,d\omega \;\; \mbox{ for all }\;\; \varphi\in C_{c}^{\infty}(\R^{n+1}).
\end{equation}
The function $u$ and the measure $\omega$ are unique modulo constant factors, and $u$ is the so called Green function
with pole at infinity and $\omega$ the harmonic measure with pole at infinity.

From now on, we will assume that $u$ is also defined in $\R^{n+1}\setminus\Omega$ and vanishes identically
here, so that $u\in C(\R^{n+1})$. 

Given a Radon measure $\mu$ in $\R^{n+1}$, its $n$-dimensional Riesz transform is defined by
$$\RR\mu(x) = c_n\int \frac{x-y}{|x-y|^{n+1}}\,d\mu(y),$$
whenever the integral makes sense. We assume that the constant $c_n$ is chosen so that $K(x):= c_n\,x/|x|^{n+1}$ coincides with
the gradient of the fundamental solution of the Laplacian.

\vv
The main result of this section is the following.

\begin{propo}\label{proporiesz}
Let $\Omega\subset\R^{n+1}$ be an unbounded NTA domain, and let $u$ and $\omega$ be the Green function
and the associated harmonic measure with pole at infinity, respectively. Suppose that for all $x\in\partial\Omega$ there exist some constants
$0<\delta<1$ and $C>0$ (both possibly depending on $x$) such that
\begin{equation}\label{eqgrow}
\omega(B(x,r))\leq C\,r^{n+\delta}\quad  \mbox{ for all $r\geq 1$.}
\end{equation}
Then we have
\begin{equation}\label{eqriesz}
\RR\omega(x) - \RR\omega(y) = \nabla u(y) - \nabla u(x) \quad  \mbox{ for all $x,y\in\R^{n+1}\setminus
\partial\Omega$.}
\end{equation}
\end{propo}

Some remarks are in order.  First, it is easy to check that if the condition \rf{eqgrow} holds for all $x\in\partial\Omega$, then it also holds for all
$x\in\R^{n+1}$ (with some constants $C,\delta$ depending also on $x$).
For the identity \rf{eqriesz} to be true, it is important to define the Riesz transform so that its kernel
is the gradient of the fundamental solution of the Laplacian, as we did above.
 On the other hand, the function $\RR\omega$ is defined modulo
a constant term (i.e., in a BMO sense). So for all $x,y\in\R^{n+1}\setminus
\partial\Omega$, by definition we write
$$\RR\omega(x) - \RR\omega(y) = \int \bigl(K(x-z) - K(y-z)\bigr)\,d\omega(z).$$
Then it turns out that the integral on the right hand side above is absolute convergent. Indeed, denoting
$d=\max(2|x-y|,1)$,
we have
\begin{align}\label{eq**0}
\int_{|x-z|\geq d} \bigl|K(x-z) - K(y-z)\bigr|\,d\omega(z) & \lesssim \int_{|x-z|\geq d} \frac{|x-y|}{|x-z|^{n+1}} \,d\omega(z)\\
& \lesssim \sum_{k\geq 0} \frac{|x-y|}{\bigl(2^kd\bigr)^{n+1}}\,\omega(B(x,2^kd))\nonumber\\
& \lesssim_x \sum_{k\geq 0} \frac{|x-y|}{\bigl(2^kd\bigr)^{n+1}}\,(2^{k}d)^{k(n+\delta)} <\infty,\nonumber
\end{align}
which implies that
$$\int \bigl|K(x-z) - K(y-z)\bigr|\,d\omega(z) <\infty,$$
since $x,y\not\in \supp\omega =\partial\Omega$.

\vv

Before proceeding with the proof of Proposition \ref{proporiesz}, we recall a few lemmas about NTA domains. These lemmas were originally shown in \cite{Jerison-Kenig} for bounded NTA domains, but as the arguments for these results are purely local, they also hold for unbounded NTA domains. 

\begin{lemma}[{\cite[Lemma 4.4]{Jerison-Kenig}}] Let $\Omega\subseteq \R^{n+1}$ be NTA and $B$ a ball centered on $\d\Omega$ with $0<r(B)<\diam \d\Omega$. Let $x_{B}$ be a corkscrew point for $B$ in $\Omega$ and let $g$ be the Green function for $\Omega$. Then
\begin{equation}\label{e:wsimg}
\omega^{z}(B)\approx g(x_B,z)r^{1-n} \mbox{ for all }z\in \Omega\backslash 2B.
\end{equation}
\end{lemma}
\vv

\begin{lemma}[{\cite[Lemma 4.10]{Jerison-Kenig}}] Let $\Omega\subseteq \R^{n+1}$ be an NTA domain and $B$ a ball centered on $\d\Omega$ with $0<Mr(B)<\diam \d\Omega$, where $M$ depends on the NTA character of $\Omega$. Suppose $u,v$ are two positive harmonic functions in $\Omega$ vanishing continuously on $MB\cap \d\Omega$ and let $x_{B}$ be a corkscrew point for $B$ in $\Omega$. Then
\begin{equation}\label{e:bharnack}
\frac{u(z)}{v(z)}\approx \frac{u(x_{B})}{v(x_{B})} \;\;  \mbox{ for all }\;\; z\in B\cap \Omega.
\end{equation}
\end{lemma}

\vv
\begin{proof}[Proof of Proposition \ref{proporiesz}]
As shown in \cite[Section 3]{Kenig-Toro-annals}, the Green function $u$ and the harmonic measure $\omega$ 
with pole at infinity can be constructed as follows. Given a fixed point $a\in \Omega$ and a sequence of points $p_j\in\Omega$ such that
$p_j\to\infty$, we consider the function
$$u_j(x) = \left\{\begin{array}{ll} \dfrac{g(x,p_j)}{g(a,p_j)}& \quad \mbox{ if $x\in\Omega$},\\{}\\
0& \quad \mbox{ if $x\not\in\overline\Omega$,}\end{array}\right.
$$
and the measure
$$\omega_j = \frac1{g(a,p_j)}\,\omega^{p_j}.$$
Passing to a subsequence and relabelling if necessary, we may assume that $u_j$ is locally uniformly convergent and that $\omega_j$ is weakly
convergent. Then it turns out that $u$ is the weak limit of the sequence $u_j$ and $\omega$ is the
weak limit of $\omega_j$. 
For simplicity, we choose points $p_j$ such that $|p_j-a|\approx \dist(p_j,\d\Omega)\to\infty$. Observe that \eqn{wsimg} and our definitions of $u$ and $\omega$, it follows that for all balls $B$ centered on $\d\Omega$, if $x_{B}$ is a corkscrew point for $B$ in $\Omega$, then
\begin{equation}\label{e:wsimu}
\omega(B)r^{1-n}\approx u(x_{B}).
\end{equation}


It is well known that the Green function $g(\cdot,\cdot)$ equals
$$g(x,p) = \EE(x-p) - \int \EE(x-z)\,d\omega^p(z)\qquad \mbox{for $x,p\in\Omega$,}$$
where $\EE$ stands for the fundamental solution of the Laplacian.
On the other hand, the right hand side above vanishes if $x\in\R^{n+1}\setminus\overline\Omega$, $p\in\Omega$. So we deduce that for all $x\not \in\partial\Omega$,
$$\nabla u_j(x) = \frac{1}{g(a,p_j)}\,K(x-p_j) - \RR\omega_j(x).$$
Thus, for all $x,y\not \in\partial\Omega$,
$$\nabla u_j(y)  - \nabla u_j(x) = \frac{1}{g(a,p_j)}\,\bigl(K(y-p_j) - K(x-p_j)\bigr)
 + \RR\omega_j(x) - \RR\omega_j(y).$$
Since $u_j$ is harmonic out of $\partial\Omega$ and $u_j$ converges locally uniformly to $u$, it turns out that $\nabla u_j$ converges also locally uniformly to $\nabla u$ out of $\partial\Omega$. Hence to prove the
proposition it suffices to show that
\begin{equation}\label{eqlim1}
\lim_{j\to\infty}\frac{1}{g(a,p_j)}\,\bigl(K(y-p_j) - K(x-p_j)\bigr)=0
\end{equation}
and
\begin{equation}\label{eqlim2}
\lim_{j\to\infty}\bigl(\RR\omega_j(x) - \RR\omega_j(y)\bigr)=\RR\omega(x) - \RR\omega(y).
\end{equation}

To prove the above identities first we will estimate $g(a,p_j)$ in terms of $u$ and $\omega$. To this end,
we will apply the boundary Harnack principle. 

Let $\xi_j\in\partial\Omega$ be such that $|\xi_j-p_j|=
\dist(p_j,\partial\Omega)$, and consider the ball $B(p_j)=B(\xi_j,|\xi_j-p_j|)$. Suppose that
 $|\xi_j-p_j|\gg \dist(a,\partial\Omega)$. 
Consider a corkscrew point $\wt p_j
\in\frac12 B(p_j)\cap \Omega$, so that $\dist(\wt p_j,\partial\Omega)\approx r(B(p_j))$. Since $u$ and $g(\cdot,p_j)$ are harmonic in $\Omega\cap B(p_j)$ and vanish identically in $\partial\Omega$, 
we deduce  from \eqn{bharnack} that
$$\frac{g(\wt p_j,p_j)}{g(a,p_j)} \approx \frac{u(\wt p_j)}{u(a)},$$
since $a$ belongs to $C\,B(p_j)$, for some fixed constant $C$, and $\dist(a,\partial\Omega)\ll r(B(p_j))$ by
assumption. Taking into account that by \eqn{wsimu}
$$u(\wt p_j)\approx u(p_j)\approx \omega(B(p_j))\,|p_j-\xi_j|^{1-n}\approx \omega(B(p_j))\,|p_j-a|^{1-n}$$
and
$$g(\wt p_j,p_j)\approx \frac1{|\wt p_j- p_j|^{n-1}} \approx 
\frac1{|p_j-a|^{n-1}},$$ we infer that
\begin{equation}\label{eqga1}
g(a,p_j) \approx \frac{u(a)}{\omega(B(p_j))}.
\end{equation}

With \rf{eqga1} at hand, we are ready to prove \rf{eqlim1}:
$$\frac{1}{g(a,p_j)}\,\bigl|K(y-p_j) - K(x-p_j)\bigr| \lesssim
\frac{\omega(B(p_j))}{u(a)}\,\frac{|x-y|}{|x-p_j|^{n+1}}.$$
For $j$ big enough, we have $r(B(p_j))\approx |x-p_j|$, and then we derive
$$\frac{\omega(B(p_j))}{|x-p_j|^{n+1}} \lesssim_x \frac{|x-p_j|^{n+\delta}}{|x-p_j|^{n+1}} = 
\frac1{|x-p_j|^{1-\delta}},$$
and thus
$$\frac{1}{g(a,p_j)}\,\bigl|K(y-p_j) - K(x-p_j)\bigr| \lesssim_x
\frac{|x-y|}{u(a)}\,\frac1{|x-p_j|^{1-\delta}} \to 0 \quad \mbox{ as $j\to\infty$.}$$

We turn our attention to the identity \rf{eqlim2} now. Take an auxiliary radial $C^\infty$ function $\vphi:\R^{n+1}\to\R$ such that $\chi_{B(0,1)}\leq \vphi\leq\chi_{B(0,2)}$ and denote $\vphi_\ve(z) = \vphi(z/\ve)$.
For $\ve>0$, we denote 
$$K_\ve = (1-\vphi_\ve)\,K\quad \mbox{ and }\quad  \wt K_\ve = \vphi_\ve\,K.$$
Notice that $K_\ve$ and $\wt K_\ve$ are standard Calder\'on-Zygmund kernels. We denote by $\RR_\ve$ and
$\wt\RR_\ve$ the respective associated operators, so that, at least formally, $\wt \RR_\ve$ tends to $\RR$ as $\ve\to\infty$. Then we write
\begin{align}\label{eq**1}
\bigl|\bigl(\RR\omega_j(x) - \RR\omega_j(y)\bigr)- \bigl(\RR\omega&(x) - \RR\omega(y)\bigr)\bigr|\\
& \leq 
\bigl|\bigl(\wt \RR_\ve\omega_j(x) - \wt\RR_\ve\omega_j(y)\bigr)- \bigl(\wt\RR_\ve\omega(x) - 
\wt\RR_\ve\omega(y)\bigr)\bigr| \nonumber
\\ &\quad
+ \bigl| \RR_\ve\omega_j(x) - \RR_\ve\omega_j(y)\bigr| + \bigl|\RR_\ve\omega(x) - 
\RR_\ve\omega(y)\bigr|.\nonumber
\end{align}
Since the function $\wt K_\ve(x-\cdot)- \wt K_\ve(y-\cdot)$ is continuous on $\partial\Omega$ (recall that $x,y\in\R^{n+1}\setminus \partial \Omega$) and has compact support, we infer that 
\begin{equation}\label{eq**1.5}
\bigl|\bigl(\wt \RR_\ve\omega_j(x) - \wt\RR_\ve\omega_j(y)\bigr)- \bigl(\wt\RR_\ve\omega(x) - 
\wt\RR_\ve\omega(y)\bigr)\bigr|\to 0 \quad \mbox{ as $j\to\infty$,}
\end{equation}
by the weak convergence of $\omega_j$ to $\omega$.

Concerning the second term on the right hand side of \rf{eq**1}, we will show below that
\begin{equation}\label{eq**2}
\bigl| \RR_\ve\omega_j(x) - \RR_\ve\omega_j(y)\bigr|\lesssim_x
\frac{|x-y|}{u(a)}\,\left(\frac1{\ve^{1-\delta}} + \frac1{|x-p_j|^{1-\delta}}\right)
\end{equation}

The last term in \rf{eq**1} is estimated as in \rf{eq**0}. Indeed, for $\ve\gg|x-y|$,
\begin{align}\label{eq**3}
\bigl|\RR_\ve\omega(x) - 
\RR_\ve\omega(y)\bigr| & \lesssim
\int \bigl| K_\ve(x-z) -  K_\ve(y-z)\bigr|\,d\omega(z)\\
 & \lesssim \int_{|x-z|\geq \ve/2} \frac{|x-y|}{|x-z|^{n+1}} \,d\omega(z)\nonumber \\
& \lesssim \sum_{k\geq 0} \frac{|x-y|}{\bigl(2^k\ve\bigr)^{n+1}}\,\omega(B(x,2^k\ve))\nonumber\\
&   \lesssim_x \sum_{k\geq 0} \frac{|x-y|}{\bigl(2^k\ve\bigr)^{n+1}}\,(2^{k}\ve)^{n+\delta}
\approx \frac{|x-y|}{\ve^{1-\delta}}.\nonumber
\end{align}

From \rf{eq**1}, \rf{eq**1.5}, \rf{eq**2} and \rf{eq**3} we deduce that
$$\limsup_{j\to\infty}\bigl|\bigl(\RR\omega_j(x) - \RR\omega_j(y)\bigr)- \bigl(\RR\omega(x) - \RR\omega(y)\bigr)\bigr| \lesssim_x \frac{|x-y|}{u(a)\,\ve^{1-\delta}} + \frac{|x-y|}{\ve^{1-\delta}}.$$
Since this holds for any arbitrarily big $\ve>0$, the limit vanishes and this concludes the proof of
the identity \rf{eqriesz}.

Finally we deal with the estimate \rf{eq**2}. Arguing as in \rf{eq**3}, with $\omega$ replaced by $\omega_j$,
we obtain
$$
\bigl| \RR_\ve\omega_j(x) - \RR_\ve\omega_j(y)\bigr|
\lesssim \sum_{k\geq 0} \frac{|x-y|}{\bigl(2^k\ve\bigr)^{n+1}}\,\omega_j(B(x,2^k\ve))
.$$ 
We split the last sum according to whether $2^k\ve \leq |p_j-x|$ or not, so that
$$\bigl| \RR_\ve\omega_j(x) - \RR_\ve\omega_j(y)\bigr|\leq S_1 + S_2,$$
where
$$S_1= \!\!\!\sum_{\substack{k\geq 0\\2^k\ve\leq |p_j-x|}}\!
 \frac{|x-y|}{\bigl(2^k\ve\bigr)^{n+1}}\,\omega_j(B(x,2^k\ve))
 \quad\! \mbox{ and } \!\quad S_2=\!\!\!\sum_{\substack{k\geq 0\\2^k\ve > |p_j-x|}}\!
 \frac{|x-y|}{\bigl(2^k\ve\bigr)^{n+1}}\,\omega_j(B(x,2^k\ve)).$$
To estimate $S_1$ we use the fact that, for $2^k\ve\leq |p_j-x|$,
$$\omega_j(B(x,2^k\ve)) = \frac1{g(a,p_j)}\,\omega^{p_j}(B(x,2^k\ve)) \approx 
\frac1{g(a,p_j)}\,\frac{\omega(B(x,2^k\ve))}{\omega(B(p_j))}. 
$$
Hence, by \rf{eqga1},
$$\omega_j(B(x,2^k\ve)) \approx\frac{\omega(B(x,2^k\ve))}{u(a)},$$
and so 
$$S_1\lesssim\sum_{k\geq 0}
 \frac{|x-y|}{\bigl(2^k\ve\bigr)^{n+1}}\,\frac{\omega(B(x,2^k\ve))}{u(a)} \lesssim_x 
 \sum_{k\geq 0}
 \frac{|x-y|}{\bigl(2^k\ve\bigr)^{n+1}}\,\frac{(2^k\ve)^{n+\delta}}{u(a)} \lesssim 
  \frac{|x-y|}{u(a)\,\ve^{1-\delta}}.$$

To estimate $S_2$ we use the following trivial estimate:
$$\omega_j(B(x,2^k\ve)) = \frac1{g(a,p_j)}\,\omega^{p_j}(B(x,2^k\ve)) \leq \frac1{g(a,p_j)} \approx
\frac{\omega(B(p_j))}{u(a)}. 
$$
Therefore,
$$ S_2\approx\!\sum_{\substack{k\geq 0\\2^k\ve > |p_j-x|}}\!
 \frac{|x-y|}{\bigl(2^k\ve\bigr)^{n+1}}\,\frac{\omega(B(p_j))}{u(a)}
\lesssim  \frac{|x-y|}{|p_j-x|^{n+1}}\,\frac{\omega(B(p_j))}{u(a)}.$$
Assuming that $|p_j-x|\geq 1$, we have
$$\omega(B(p_j))\lesssim_x r(B(p_j))^{n+\delta}\approx |p_j-x|^{n+\delta},$$
and thus
$$S_2\lesssim_x \frac{|x-y|}{|p_j-x|^{1-\delta}}\,\frac{\omega(B(p_j))}{u(a)}.$$
From this estimate and the one for $S_1$, we obtain \rf{eq**2}, as wished. 
\end{proof}

\vv
We recall now the following version of the jump equations for the gradient of the single layer potential due to Hofmann, Mitrea and Taylor \cite{HMT}:

\begin{propo}[{\cite[Proposition 3.30]{HMT}}]\label{propoplem}
Let $\Omega\subset\R^{n+1}$ be a domain in $\R^{n+1}$ with uniformly rectifiable boundary such that 
$\sigma(\partial\Omega\setminus \partial_*\Omega)=0$, where $\partial_*\Omega$ stands for the measure theoretic boundary and $\sigma$ for the surface measure of $\Omega$.
Let $f\in L^p(\sigma|_{\partial\Omega})$, for $1<p<\infty$. Then, for $\sigma$-a.e.\ $x\in\partial\Omega$,
\begin{equation}\label{eqplem1}
\lim_{\Gamma^-(x)\ni z\to x} \RR (f\sigma)(z) = -\frac12\,\vec{n}(x)\,f(x) +{\rm pv}\,\RR (f\sigma)(x)
\end{equation}
and
\begin{equation}\label{eqplem2}
\lim_{\Gamma^+(x)\ni z\to x} \RR(f\sigma)(z) = \frac12\,\vec{n}(x)\,f(x) +{\rm pv}\,\RR (f\sigma)(x)
\end{equation}
where $\Gamma^+(x)$ is a non-tangential cone at $x$ relative to $\Omega$, (that is,
\[
\Gamma^{+}(x)=\{y\in \Omega: \dist(y,\Omega^{c})>t |y-x|\}\]
for some $t>0$), $\Gamma^-(x)$ is a non-tangential cone at $x$ relative to 
$\R^{n+1}\setminus \overline\Omega$, and $\vec{n}(x)$ is the outer normal to $\Omega$ at $x$.
\end{propo}

In particular, if $\Omega$ is a chord-arc domain in $\R^{n+1}$, then $\partial\Omega$ is uniformly rectifiable (see \cite{DJ}) and $\sigma(\partial\Omega\setminus \partial_*\Omega)=0$, thus the preceding proposition can be applied.

\vv

\begin{propo}\label{proponontang}
Let $\Omega\subset\R^{n+1}$ be a chord-arc domain in $\R^{n+1}$. Let $\omega$ and $u$ be the harmonic measure and the Green
function with a pole either at infinity or at some point $p\in\Omega$. 
Suppose that for each $x\in\partial\Omega$ there exist some constants
$0<\delta<1$ and $C>0$ such that
\begin{equation}\label{eqgrow'}
\omega(B(x,r))\leq C\,r^{n+\delta}\quad  \mbox{ for all  $r\geq 1$.}
\end{equation}
Suppose that $h:=\dfrac{d\omega}{d\sigma}\in L^p_{loc}(\sigma)$ for some $p>1$.
Then  $\lim_{\Gamma^+(x)\ni z\to x} \nabla u(z)$ exists for $\sigma$-a.e.\ $x\in\partial\Omega$ and
\begin{equation}\label{eq:n.t.limit}
\lim_{\Gamma^+(x)\ni z\to x} \nabla u(z) = -h(x)\,\vec{n}(x).
\end{equation} 
\end{propo}

\begin{proof}
Assume that the pole for $\omega$ and $u$ is at infinity (the arguments for the case where the pole is finite are analogous).
Let $B$ be a ball centered at $\partial\Omega$. By Proposition \ref{propoplem}, for $\sigma$-a.e.\ $x\in B$,
$$
\lim_{\Gamma^-(x)\ni z\to x} \RR (\chi_{2B}\omega)(z) = -\frac12\,\vec{n}(x)\,h(x) +{\rm pv}\RR (\chi_{2B}\omega)
$$
and
$$
\lim_{\Gamma^+(x)\ni z\to x} \RR (\chi_{2B}\omega)(z) = \frac12\,\vec{n}(x)\,h(x) +{\rm pv}\RR (\chi_{2B}\omega)
$$
In particular,
$$\lim_{\Gamma^+(x)\ni z\to x} \RR (\chi_{2B}\omega)(z) - \lim_{\Gamma^-(x)\ni z\to x} \RR (\chi_{2B}\omega)(z) = \vec{n}(x)\,h(x).$$
Using the condition \rf{eqgrow'}, by estimates analogous to the ones in \rf{eq**0},
it is immediate to check that
$$\lim_{\Gamma^+(x)\ni z\to x} \RR (\chi_{2B}\omega)(z) - \lim_{\Gamma^-(x)\ni z\to x} \RR (\chi_{2B}\omega)(z) = 
\lim_{\Gamma^+(x)\ni z\to x} \RR \omega(z) - \lim_{\Gamma^-(x)\ni z\to x} \RR \omega(z).$$
Then, by Proposition \ref{proporiesz} we infer that
$$\lim_{\Gamma^-(x)\ni z\to x} \nabla u (z) - \lim_{\Gamma^+(x)\ni z\to x} \nabla u(z) = \vec{n}(x)\,h(x).$$
Since $u\equiv 0$ in $\R^{n+1}\setminus \overline \Omega$, we have $\lim_{\Gamma^-(x)\ni z\to x} \nabla u (z)=0$ and so
$$- \lim_{\Gamma^+(x)\ni z\to x} \nabla u(z) = \vec{n}(x)\,h(x) \quad \mbox{for $\sigma$-a.e.\ $x\in \partial\Omega\cap B$.}$$
\end{proof}

\vvv


\section{Some technical lemmas}\label{sec3}\label{section4}

From now on, given a domain $\Omega\subset\R^{n+1}$ and $x\in\R^{n+1}$, we denote
$$d_\Omega(x) = \dist(x,\Omega^c).$$

The following is a well known result. See for example \cite[Section 4]{Jerison-Kenig}.

\begin{lemma}\label{lemholder}
Let $\Omega\subset\R^{n+1}$ be an NTA domain and let $B$ be a ball centered at $\partial\Omega$. There exist some constants $C,\alpha>0$ depending on the NTA character of $\Omega$ such that the following holds.
If $u$ is a non-negative harmonic function on $\Omega\cap 2B$ which vanishes continuously on $\partial \Omega\cap2B$, then
$$u(x)\leq C\,\left(\frac{d_\Omega(x)}{r(B)}\right)^\alpha\,\sup_{y\in\partial( 2B)\cap\Omega}u(y) \quad\mbox{ for all $x\in B\cap\Omega$}.$$
If $x_{B}$ is a corkscrew point for $B$, then
\[
\sup_{y\in B\cap\Omega}u(y) \leq C u(x_{B}).\]
\end{lemma}


\vv
We will also need the next auxiliary result.

\begin{lemma}\label{lemgreen*}
Let $\Omega\subset\R^{n+1}$ be an NTA domain. There exists some constants $C,\alpha>0$ depending on the NTA character of $\Omega$ such that the Green function of $\Omega$ satisfies:
\begin{equation}\label{eqgreen*}
g(x,y)\leq C\,\frac1{|x-y|^{n-1}}\,\left(\frac{\min\bigl(d_\Omega(x),d_\Omega(y)\bigr)}{|x-y|}\right)^\alpha 
\quad \mbox{ for all $x,y\in\Omega$.}
\end{equation}
\end{lemma}

\begin{proof}
It is enough to show that, for some $C,\alpha'>0$,
\begin{equation}\label{eqgreen*2}
g(x,y)\leq C\,\frac1{|x-y|^{n-1}}\,\left(\frac{d_\Omega(x)}{|x-y|}\right)^{\alpha}
\quad \mbox{ for all $x,y\in\Omega$,}
\end{equation}
because then the analogous inequality interchanging $x$ by $y$ also holds, by symmetry. 

Because of the trivial estimate $g(x,y)\lesssim 1/|x-y|^{n-1}$, to prove \rf{eqgreen*2} we may assume that $|x-y|>10\,d_\Omega(x)$. Let $\xi_x\in\partial\Omega$ be such that $|\xi_x-x| = d_\Omega(x)$ and consider the ball $B:=B(\xi_x,|x-y|/8)$. 
Clearly $x\in B$, as 
$$|x-\xi_x|=d_\Omega(x)\leq \frac1{10}\,|x-y| = \frac8{10}\,r(B).$$
Note also that, for all $z\in\partial(2B)$,
$$|y-z|\geq |x-y|-|x-z| \geq 8\,r(B) - 4\,r(B)= 4\,r(B)\approx |x-y|.$$
Hence $g(z,y)\lesssim 1/|y-z|^{n-1}\lesssim 1/|x-y|^{n-1}$ for all $z\in\partial(2B)$. Thus, \rf{eqgreen*2} follows from Lemma \ref{lemholder} applied to the function $g(\cdot,y)$.
\end{proof}

\vv

The following rather standard result is shown in \cite[Theorem 2.1]{Kenig-Toro2}.

\begin{lemma}\label{lemvmo}
Let $\Omega\subset\R^{n+1}$ be a chord-arc domain, $f\in \textrm{VMO}(\sigma)$, and $h=e^f$. Then, for all $x\in
\partial\Omega$, $0<r\leq\diam(\Omega)$ and $1<p<\infty$,
$$\left(\,\avint_{B(x,r)} h^p\,d\sigma\right)^{1/p} \leq C_p \;\avint_{B(x,r)} h\,d\sigma$$
and
$$\left(\,\avint_{B(x,r)} h^{-p}\,d\sigma\right)^{1/p} \leq C_p \avint_{B(x,r)} h^{-1}\,d\sigma.$$
\end{lemma}

\vv

The next lemma is proven in \cite[Lemma 4.11]{Jerison-Kenig}:
\begin{lemma}
Let $\Omega$ be an NTA domain, $B$ a ball centered on $\d\Omega$ with $0<r(B)<\diam \d\Omega$, and let $E\subseteq B\cap \d\Omega$ be Borel. If $x_B$ is a corkscrew point for $B$ in $\Omega$, then
\begin{equation}
\frac{\omega^{z}(E)}{\omega^{z}(B)} \approx \omega^{x_{B}}(E) \;\; \mbox{ for }z\in \Omega\backslash 2B.
\label{e:wratio}
\end{equation}
\end{lemma}
Note that this implies that if $\omega$ is harmonic measure with pole at infinity, we also have 
\begin{equation}
\frac{\omega(E)}{\omega(B)} \approx \omega^{x_{B}}(E).
\label{e:wratio2}
\end{equation}
\vv

The next corollary is an easy consequence of the preceding lemma, as shown in  \cite[Corollary 2.4]{Kenig-Toro2}.

\begin{coro}\label{corovmo0}
Let $\Omega\subset \R^{n+1}$ be a chord-arc domain. If the harmonic measure $\omega$ in $\Omega$ 
is such that $\frac{d\omega}{d\sigma}\in \textrm{VMO}(\sigma)$, then, for all $\ve>0$, $x\in
\partial\Omega$, $0<r\leq\diam(\Omega)$ and $E\subset B(x,r)\cap\partial\Omega$,
$$C(\ve)^{-1}\left(\frac{\sigma(E)}{\sigma(B(x,r))}\right)^{1+\ve}\leq
\frac{\omega(E)}{\omega(B(x,r))} \leq C(\ve)\left(\frac{\sigma(E)}{\sigma(B(x,r))}\right)^{1-\ve}.
$$
\end{coro}

Let us remark that the pole of harmonic measure above can be either a point $p\in\Omega$ (in which case the constants also depend on $p$) or infinity in case
$\Omega$ is unbounded.

Another easy consequence of Lemma \ref{lemvmo} is the following.

\begin{coro}\label{corovmo}
Let $\Omega\subset \R^{n+1}$ be a chord-arc domain. Suppose that the  harmonic measure $\omega$ in $\Omega$ with pole at infinity
is such that $\log\left(\frac{d\omega}{d\sigma}\right)\in \textrm{VMO}(\sigma)$. For $z\in\Omega$, let $K_z = \frac{d\omega^z}{d\sigma}$ (i.e., $K_z$ is the Poisson kernel). For $1<p<\infty$, if $x\in\partial\Omega$, $0<r\leq\diam(\Omega)$, and $z\in\Omega\setminus B(x,2r)$, then
$$\left(\,\avint_{B(x,r)} (K_z)^p\,d\sigma\right)^{1/p} \leq C(p)\, \avint_{B(x,r)} K_z\,d\sigma.$$
\end{coro}

For this corollary to hold we assume either that the pole of $\omega$ is at $\infty$ if $\Omega$ is unbounded, or that it is in $\Omega$.

\begin{proof}
Since $z\in\Omega\setminus B(x,2r)$, if $z_{0}$ is a corkscrew point for $B(x,r)$, then whenever $B(y,s)\subset  B(x,r)$ and all $0<s<r/10$, by \eqn{wratio} and \eqn{wratio2},
$$  \frac{\omega(B(y,s))}{\omega(B(x,r))} \approx \omega^{z_{0}}(B(y,s))\approx \frac{\omega^{z}(B(y,s))}{\omega^{z}(B(x,r))}.$$

Hence, by the Lebesgue differentiation theorem, if we denote $h=\frac{d\omega}{d\sigma}$, 
for $\sigma$-a.e. $y\in B(x,r)\cap\partial\Omega$,
\begin{align*}
K_z(y) & = \frac{d\omega^z}{d\sigma}(y)  = \lim_{s\to 0}\frac{\omega^z(B(y,s))}{\sigma(B(y,s))}\\
& \approx \frac{\omega^{z}(B(x,r))}{\omega(B(x,r))}\,\lim_{s\to 0}\frac{\omega(B(y,s))}{\sigma(B(y,s))}
= \frac{\omega^{z}(B(x,r))}{\omega(B(x,r))} h(y).
\end{align*}
Therefore, by Lemma \ref{lemvmo}, since $\log h\in \textrm{VMO}(\sigma)$,
\begin{align*}
\left(\,\avint_{B(x,r)} K_z(y)^p\,d\sigma(y)\right)^{1/p} &\approx  \frac{\omega^{z}(B(x,r))}{\omega(B(x,r))}  \left(
\,\,\avint_{B(x,r)} h(y)^p \,d\sigma(y)\right)^{1/p}\\
& \lesssim  \frac{\omega^{z}(B(x,r))}{\omega(B(x,r))} 
\,\,\avint_{B(x,r)} h(y) \,d\sigma(y)\\
& \approx \,\avint_{B(x,r)} K_z(y)\,d\sigma(y).
\end{align*}
\end{proof}

\vv

\begin{lemma}\label{lemnontang}
Let $\Omega\subset \R^{n+1}$ be a chord-arc domain. Suppose that the  harmonic measure $\omega$ in $\Omega$
 with pole either at infinity or at some fixed point $p\in\Omega$ 
is such that $\log\left(\frac{d\omega}{d\sigma}\right)\in \textrm{VMO}(\sigma)$.
Denote by $u$ the associated Green function.
Then, for $\sigma$-a.e.\ $x\in\partial\Omega$, $\nabla u(z)$ converges to $-\vec{n}(x)\,\frac{d\omega}{d\sigma}(x)$  as $\Omega\ni z\to
x$ non-tangentially, where $\vec n$ is the outer unit normal of $\Omega$.
\end{lemma}

This lemma is proved in \cite{Kenig-Toro2} under the additional assumption that
$\Omega$ is Reifenberg flat. In \cite{Kenig-Toro} it is shown how to prove this without the Reifenberg flatness assumption.
The delicate arguments involved  in \cite{Kenig-Toro2} and  \cite{Kenig-Toro} do not use the connection
between harmonic measure and the Riesz transform and instead are of a more geometric nature.

\begin{proof}
This is an immediate consequence of Proposition \ref{proponontang} and Corollary
\ref{corovmo0}. Indeed, this corollary, implies that for all $x\in\partial\Omega$ and all
$0<r_0\leq r\leq\diam(\Omega)$,
$$\left(\frac{\sigma(B(x,r_0))}{\sigma(B(x,r))}\right)^{1+\ve}\leq C(\ve)\,
\frac{\omega(B(x,r_0))}{\omega(B(x,r))},
$$
Hence, using also the AD-regularity of $\sigma$ we get
$$\omega(B(x,r)) \leq  C(\ve)\,\omega(B(x,r_0))\,\left(\frac{\sigma(B(x,r))}{\sigma(B(x,r_0))}\right)^{1+\ve}
 \approx \frac{\omega(B(x,r_0))}{\sigma(B(x,r_0))^{1+\ve}}
\,r^{n(1+\ve)}.$$
Therefore, choosing $\ve=1/(2n)$,
$$\omega(B(x,r)) \leq C(x)\,r^{n+1/2}\quad \mbox{ for $r\geq r_0$.}$$
So the assumption \rf{eqgrow'} in Proposition \ref{proponontang} holds and thus
$$\lim_{\Gamma^+(x)\ni z\to x} \nabla u(z) = -\dfrac{d\omega}{d\sigma}(x)\,\vec{n}(x)\quad \mbox{ for $\sigma$-a.e.\
$x\in\partial\Omega$.}
$$
\end{proof}
\vv

{ The next result is an auxiliary calculation which will be used several times in the next section.
The arguments for the proof are quite standard. Similar calculations appear, for example, in the proofs of Lemma 5.2 of \cite{Kenig-Toro}, Lemma 3.3 of \cite{Kenig-Toro2} or Lemma 3.30 of \cite{Hofmann-Martell}.}

\begin{lemma}\label{lemsumw}
Let $\Omega\subset \R^{n+1}$ be a chord-arc domain, and let $\omega$ be the harmonic measure in $\Omega$
 with pole either at infinity or at some fixed point $p\in\Omega$. Let $B\subset\R^{n+1}$ be a ball centered at $\partial\Omega$
 such that $p\not\in10B$. Then for any constant $\ve>0$,
$$\int_{B\cap\Omega} \!\left(\frac{d_\Omega(y)}{r(B)}\right)^{\ve}
\frac{\omega(B(y,2d_\Omega(y)))}{d_\Omega(y)^{n+1}}\,dy \leq C(\ve)\,\omega(B).$$
\end{lemma}

\begin{proof}
We write
\begin{align}\label{eqplug*23}
\int_{B\cap\Omega} \!\left(\frac{d_\Omega(y)}{r(B)}\right)^{\ve}&
\frac{\omega(B(y,2d_\Omega(y)))}{d_\Omega(y)^{n+1}}\,dy \\
&
\lesssim \sum_{j\geq0} 2^{-j\ve}
\int_{\begin{subarray}{l}
y\in B\cap\Omega:\\
2^{-j-1}r(B)<d_\Omega(y)\leq 2^{-j}r(B)\end{subarray}}
\frac{\omega(B(y,2^{-j+1}r(B)))}{(2^{-j}r(B))^{n+1}}\,dy.\nonumber
\end{align}
We denote $A_j:=\bigl\{y\in B\cap\Omega:2^{-j-1}r(B)<d_\Omega(y)\leq 2^{-j}r(B)\bigr\}$. For each $y\in A_j$ consider a ball
$B_y^j$ with radius $2^{-j+1}r(B)$ centered at a point $\xi_y\in\partial\Omega$ such that $|y-\xi_y|=d_\Omega(y)$. Clearly $y\in
B_y^j$ for each $y\in A_j$, and thus we can extract a subfamily of pairwise disjoint balls $\{B_k^j\}_{k}\subset \{B_y^j\}_{y\in A_j}$
so that 
$$A_j \subset \bigcup_k 3B_k^j.$$
Notice that for each $y\in B_k^j$, since $\omega$ is doubling,
$$\omega(B(y,2^{-j+1}r(B))) \leq \omega(6B_k^j) \lesssim \omega(B_k^j).$$
Therefore, taking also into account that the balls $B_k^j$ are contained in $6B$,
\begin{align*}
\int_{\begin{subarray}{l}
y\in B\cap\Omega:\\
2^{-j-1}r(B)<d_\Omega(y)\leq 2^{-j}r(B)\end{subarray}}
\frac{\omega(B(y,2^{-j+1}r(B)))}{(2^{-j}r(B))^{n+1}}\,dy\lesssim & \sum_k 
\int_{B_k^j}
\frac{\omega(B_k^j)}{(2^{-j}r(B))^{n+1}}\,dy \\
& = C  \sum_k \omega(B_k^j) \lesssim \omega(6B) \lesssim \omega(B).
\end{align*}
Plugging this estimate into \rf{eqplug*23}, the lemma follows.
\end{proof}

\vvv


\section{Estimates for the gradient of Green's function}\label{section5}

{ The reader should compare the arguments in this 
section to the ones in Section 3 of \cite{Kenig-Toro2} and Section 2 of \cite{Kenig-Toro}, which in turn rely on the results in the Appendices A1 and A2 of 
\cite{Kenig-Toro2}. 
} 
\vv

\begin{lemma}\label{lemgradgreen1}
Let $\Omega\subset \R^{n+1}$ be an unbounded chord-arc domain. Suppose that the  harmonic measure $\omega$ in $\Omega$  with pole  at infinity satisfies $\log\left(\frac{d\omega}{d\sigma}\right)\in \textrm{VMO}(\sigma)$.
Denote by $u$ the associated Green function. Then
\begin{equation}\label{eqgrad1}
|\nabla u(x)|\leq \int_{\partial\Omega} \frac{d\omega}{d\sigma}(y)\,d\omega^x(y)
\quad\mbox{ for all $x\in\Omega$.}
\end{equation}
\end{lemma}
\vv

 The proof of this lemma would be quite immediate if the function
$\frac{d\omega}{d\sigma}$ inside the integral in \rf{eqgrad1} were compactly supported, 
taking into account that $\nabla u$ is harmonic. However, this is not
the case and so the arguments are more delicate. The next auxiliary lemma will be used to take care of this question
by a localization of singularities technique.

\begin{lemma}\label{lemtec1}
Under the assumptions of Lemma \ref{lemgradgreen1}, suppose that $0\in\partial\Omega$. Fix $R>1$ large and let $\vphi_R\in C_c^\infty(\R^{n+1})$ such that $\chi_{B(0,R)}\leq \vphi_R\leq \chi_{B(0,2R)}$, $|\nabla^j
\vphi_R|\lesssim 1/R^j$ for $j=1,2$. For $x\in\Omega$, define
$$w_R(x) = \int_\Omega g(x,y)\,\Delta[\vphi_R\nabla u ](y)\,dy.$$
Then $w_R\in C^{\alpha/2}(\overline \Omega)$ for some $\alpha>0$, $w_R|_{\partial\Omega}\equiv 0$, 
and the following estimates hold for $x\in\Omega$:
\begin{itemize}
\item[(a)] $\ds |w_R(x)| \lesssim \frac{\omega(B(0,R))}{R^n}\,\left(\frac{d_\Omega(x)}R\right)^{\alpha/2}\;\;$ if $|x|\leq 4R$.
\vv

\item[(b)]  $\ds |w_R(x)| \lesssim \frac{\omega(B(0,R))}{|x|^{n-1+\alpha/2}\,R^{1-\alpha/2}}\,\left(\frac{d_\Omega(x)}{|x|}\right)^{\alpha/2}\;\;$ if $|x| > 4R$.

\end{itemize}
\end{lemma}


\begin{proof}
By the relationship between Green's function and harmonic measure, for all $y\in\Omega$ we have
$$u(y)\approx \frac1{d_\Omega(y)^{n-1}}\,\omega(B(y,2d_\Omega(y))),$$
and by standard estimates for positive harmonic functions we derive
$$|\nabla u(y)|\lesssim \frac{u(y)}{d_\Omega(y)} \approx \frac{\omega(B(y,2d_\Omega(y)))}{d_\Omega(y)^n}
\quad \mbox{ and }\quad
|\nabla^2 u(y)|\lesssim \frac{u(y)}{d_\Omega(y)^2} \approx \frac{\omega(B(y,2d_\Omega(y)))}{d_\Omega(y)^{n+1}}
.$$
Thus,
\begin{align}\label{eqdjk28}
|w_R(x)| & = \left|\int_\Omega g(x,y)\,\bigl(\Delta\vphi_R(y)\,\nabla u (y) + 2\nabla\vphi_R(y) \cdot \nabla^2 u(y)\bigr)\,dy\right|\\
& \lesssim \int_{A(0,R,2R)\cap\Omega} g(x,y)\,\left( \frac{\omega(B(y,2d_\Omega(y)))}{R^2\,d_\Omega(y)^n} + \frac{\omega(B(y,2d_\Omega(y)))}{R\,d_\Omega(y)^{n+1}}\right)\,dy\nonumber\\
& \lesssim \int_{B(0,2R)\cap\Omega} g(x,y)\,\frac{\omega(B(y,2d_\Omega(y)))}{R\,d_\Omega(y)^{n+1}}\,dy.\nonumber
\end{align}

\vv
\noi{\bf Case 1:} $|x|\leq 4R$.\\
We split the integral on the right hand side of \rf{eqdjk28} as follows:
\begin{align}\label{eqspl1}
|w_R(x)| & \lesssim \int_{|y-x|\leq d_\Omega(x)/2} g(x,y)\,\frac{\omega(B(y,2d_\Omega(y)))}{R\,d_\Omega(y)^{n+1}}\,dy \\
&\quad+
\int_{\substack{y\in B(0,2R)\cap\Omega:\\ |y-x|> d_\Omega(x)/2}} g(x,y)\,\frac{\omega(B(y,2d_\Omega(y)))}{R\,d_\Omega(y)^{n+1}}\,dy
 =: I_1+ I_2.\nonumber
\end{align}

First we will deal with $I_1$. In the domain of integration of $I_1$ we have $d_\Omega(y)\approx d_\Omega(x)$. Taking into account that $\omega$ is doubling, in this case we derive $\omega(B(y,2d_\Omega(y)))\approx \omega(B(x,2d_\Omega(x)))$. Then using also 
 the trivial estimate $g(x,y)\lesssim1/|x-y|^{n-1}$, we get
$$I_1\lesssim \int_{|y-x|\leq d_\Omega(x)/2} \frac1{|x-y|^{n-1}}\,\frac{\omega(B(x,2d_\Omega(x))}{R\,d_\Omega(x)^{n+1}}\,dy
\approx \frac{\omega(B(x,2d_\Omega(x))}{R\,d_\Omega(x)^{n-1}}
.$$
Notice that, by Lemma \ref{lemholder},
\begin{equation}\label{eqlik32}
u(x)\lesssim \left(\frac{d_\Omega(x)}{R}\right)^\alpha\,\sup_{y\in\partial B(0,8R)\cap\Omega}u(y) \lesssim
\left(\frac{d_\Omega(x)}{R}\right)^\alpha\,u(x_R),
\end{equation}
where $x_R$ is a corkscrew point for $B(0,R)$. That is, $x_R\in B(0,R)\cap \Omega$ and $d_\Omega(x_R)\approx R$. Hence
using that $\omega(B(z,2d_\Omega(z)))\approx u(z)\,d_\Omega(z)^{n-1}$ both for $z=x$ and $z=x_R$, we deduce that
\begin{equation}\label{eqi190}
I_1\lesssim\frac{\omega(B(x,2d_\Omega(x)))}{R\,d_\Omega(x)^{n-1}} 
\lesssim
\left(\frac{d_\Omega(x)}{R}\right)^\alpha\,\frac{\omega(B(x_R,2d_\Omega(x_R)))}{R\,d_\Omega(x_R)^{n-1}} \approx 
\left(\frac{d_\Omega(x)}{R}\right)^\alpha\,\frac{\omega(B(0,R))}{R^n}.
\end{equation}

\vv

We consider now the integral $I_2$ in \rf{eqspl1}. To estimate this we use the inequality
\begin{equation}\label{eqgreen**}
g(x,y)\lesssim \frac1{|x-y|^{n-1}}\,\left(\frac{d_\Omega(x)}{|x-y|}\right)^{\alpha/2} \,\left(\frac{d_\Omega(y)}{|x-y|}\right)^{\alpha/2},
\end{equation}
which is an immediate consequence of \rf{eqgreen*}.
To shorten notation, for each integer $j\geq0$ we write $r_j:=2^j\,d_\Omega(x)$. 
Denote by $j_{\max}$ the least integer such that $B(0,2R)\subset B(x,r_{j_{\max}})$, so that $r_{j_{\max}}\approx R$.
Then plugging the estimate \rf{eqgreen**} inside $I_2$ and splitting, we obtain 
$$I_2\lesssim 
\sum_{0\leq j \leq j_{\max}} \frac1{R\,r_j^{n-1}}\left(\frac{d_\Omega(x)}{r_j}\right)^{\alpha/2} \!
\int_{\substack{y\in \Omega: \,r_{j-1}<|y-x|\leq r_j}} \!\left(\frac{d_\Omega(y)}{r_j}\right)^{\alpha/2}
\frac{\omega(B(y,2d_\Omega(y)))}{d_\Omega(y)^{n+1}}\,dy.$$
Let $\xi_x\in\partial\Omega$ be such that $|x-\xi_x|=d_\Omega(x)$. It is immediate to check that
if $|y-x|\leq r_j=2^jd_\Omega(x)$, then $y\in \overline B(\xi_x,2r_j)$. So the last integral is bounded above by
$$\int_{\Omega\cap \overline B(\xi_x,2r_j)} \!\left(\frac{d_\Omega(y)}{r_j}\right)^{\alpha/2}
\frac{\omega(B(y,2d_\Omega(y)))}{d_\Omega(y)^{n+1}}\,dy,$$
and then, by Lemma \ref{lemsumw}, this does not exceed $C\,\omega(B(\xi_x,r_j))$. Hence,
\begin{equation}\label{eqdk472}
I_2\lesssim 
\sum_{0\leq j \leq j_{\max}} \frac1{R\,r_j^{n-1}}\left(\frac{d_\Omega(x)}{r_j}\right)^{\alpha/2} \omega(B(\xi_x,r_j)).
\end{equation}

To estimate the right hand in the inequality above we argue as in \rf{eqlik32}. We consider a corkscrew point $x_j$ in each ball 
$B(\xi_x,r_j)$, and then since $\dist(x_j,\partial\Omega)\approx r_j$, we deduce
$$u(x_j)  \lesssim
\left(\frac{r_j}{R}\right)^\alpha\,u(x_R)$$
(recall that $x_R$ is a corkscrew point for $B(0,R)$). Thus,
$$\frac{\omega(B(\xi_x,r_j))}{r_j^{n-1}}\,  \lesssim
\left(\frac{r_j}{R}\right)^\alpha\,\frac{\omega(B(0,R))}{R^{n-1}}.$$
Plugging this estimate into \rf{eqdk472} we obtain
\begin{align*}
I_2& \lesssim 
\sum_{0\leq j \leq j_{\max}} \left(\frac{d_\Omega(x)}{r_j}\right)^{\alpha/2} \left(\frac{r_j}{R}\right)^\alpha\,\frac{\omega(B(0,R))}{R^{n}}\\
& = \frac{d_\Omega(x)^{\alpha/2}}{R^\alpha}\,\frac{\omega(B(0,R))}{R^{n}}
\sum_{0\leq j \leq j_{\max}} r_j^{\alpha/2}.
\end{align*}
Since the last sum is geometric, it turns out that
$$\sum_{0\leq j \leq j_{\max}} r_j^{\alpha/2} \approx r_{j_{\max}}^{\alpha/2} \approx R^{\alpha/2}.$$
Therefore,
$$I_2 \lesssim  \frac{d_\Omega(x)^{\alpha/2}}{R^{\alpha/2}}\,\frac{\omega(B(0,R))}{R^{n}}.$$
Together with the estimate for $I_1$ in \rf{eqi190} this yields the inequality (a) in the lemma.
\vv

\noi{\bf Case 2:} $|x|> 4R$.\\
To estimate the integral on the right hand side of \rf{eqdjk28}
we use the fact that, for $y\in B(0,2R)$, by \rf{eqgreen*},
$$
g(x,y)\lesssim\frac1{|x|^{n-1}}\,\left(\frac{d_\Omega(x)}{|x|}\right)^{\alpha/2} \,\left(\frac{d_\Omega(y)}{|x|}\right)^{\alpha/2},$$
taking into account that $|x-y|\approx|x|$. Then we get
\begin{align*}
|&w_R(x)|  \lesssim \frac1{|x|^{n-1}}\,\left(\frac{d_\Omega(x)}{|x|}\right)^{\alpha/2}\int_{B(0,2R)\cap\Omega} \left(\frac{d_\Omega(y)}{|x|}\right)^{\alpha/2}\,\frac{\omega(B(y,2d_\Omega(y)))}{R\,d_\Omega(y)^{n+1}}\,dy\\
& =
\frac1{R|x|^{n-1}}\,\left(\frac{d_\Omega(x)}{|x|}\right)^{\alpha/2}\,\left(\frac{R}{|x|}\right)^{\alpha/2}
\int_{B(0,2R)\cap\Omega} \left(\frac{d_\Omega(y)}{R}\right)^{\alpha/2}\,\frac{\omega(B(y,2d_\Omega(y)))}{d_\Omega(y)^{n+1}}\,dy
.
\end{align*}
By Lemma \ref{lemsumw}, the last integral above does not exceed $C\,\omega(B(0,R))$, and so we deduce that
$$|w_R(x)|  \lesssim 
\frac1{R|x|^{n-1}}\,\left(\frac{d_\Omega(x)}{|x|}\right)^{\alpha/2}\,\left(\frac{R}{|x|}\right)^{\alpha/2}
\omega(B(0,R))
,
$$
which gives the inequality (b) in the lemma.
\end{proof}

\vv

\begin{proof}[\bf Proof of Lemma \ref{lemgradgreen1}]
The arguments are similar to  the ones for \cite[Theorem 3.1]{Kenig-Toro2}. For the reader's convenience, we show the details below.

Suppose that $0\in\partial\Omega$ and, for $R\geq1$, let $\vphi_R$ and  $w_R$ the functions introduced in Lemma
\ref{lemtec1}. For $x\in\Omega$, we define
$$h_R(x) = \vphi_R(x)\nabla u(x) - w_R(x).$$
Since $\Delta w_R= \Delta[\vphi_R\nabla u ]$ in $\Omega$, it turns out that $h_R$ is harmonic in $\Omega$. Further, the 
estimates (a) and (b) in Lemma
\ref{lemtec1}, in particular, ensure that $w_R$ vanishes continuously at $\partial\Omega$. Thus $h_R$ vanishes on $\partial
\Omega\setminus B(0,2R)$.

By Lemma \ref{lemnontang} it follows that, for $\sigma$-a.e.\ $y\in\partial\Omega$, $\nabla u(z)$
converges non-tangentially to $-\frac{d\omega}{d\sigma}(y)\vec{n}(y)$ as $\Omega\ni z\to y$. Also, as mentioned above, $w_R(z)\to0$ as 
$z\to y$. Therefore, if we denote
$$h(y)=\frac{d\omega}{d\sigma}(y),$$
we have 
$$\lim_{\Gamma^+(y)\ni z\to y} h_R(z) = -\vphi_R(y)\,h(y)\,\vec{n}(y)\quad
\mbox{ for $\sigma$-a.e.\ $y\in\partial\Omega$.}$$
We claim that for all $x\in\Omega$,
\begin{equation}\label{eqclau*20}
h_R(x) = -\int \vphi_R(y)\,h(y)\,\vec{n}(y)\,d\omega^x(y).
\end{equation}
To prove this, recalling that $h_R$ vanishes at $\infty$,
 by Theorem 5.8 and Lemma 8.3 in \cite{Jerison-Kenig} it suffices to show that
$\NN_1 h_R\in L^1(\omega^x)$ for all $x\in\Omega$, where $\NN_1$ stands for the operator defined by
$$\NN_1 h_R(y) = \sup_{z\in\Gamma_1^+(y)} h_R(z),$$
with $\Gamma_1^+(y) = \Gamma^+(y)\cap \overline B(y,1)$. By Lemma \ref{lemtec1}, $w_R$ is bounded, and
thus $\NN_1 w_R\in L^1(\omega^x)$. Hence it is enough to prove that 
$\NN_1(\vphi_R\nabla u)\in L^1(\omega^x)$. To this end, notice that if $z\in\Gamma_1^+(y)$, then
$$|\nabla u(z)| \lesssim \frac{u(z)}{d_\Omega(z)} \approx \frac{\omega(B(y,d_\Omega(z)))}{d_\Omega(z)^n}.$$
Thus, 
$$\NN_1(\vphi_R\nabla u)(y) \lesssim \sup_{0<r\leq 1} \frac{\omega(B(y,r))}{r^n} =
\sup_{0<r\leq 1} \frac1{r^n}\int_{B(y,r)}|h|\,d\sigma =:
\MM_1h(y).$$
Also, $\NN_1(\vphi_R\nabla u)(y)$ vanishes out of $B':=\overline B(0,2R+1)$ because in this case
$\vphi_R(z)=0$ whenever $z\in\Gamma_1^+(y)$. Therefore, 
$$\int \!\NN_1(\vphi_R\nabla u)\,d\omega^x = \int_{B'} \NN_1(\vphi_R\nabla u)\,K_x\,d\sigma
\lesssim \left(\int_{B'} |\MM_1 h|^2\,d\sigma\right)^{1/2}\! \left(\int_{B'} (K_x)^2\,d\sigma\right)^{1/2}\!.
$$
By the $L^2(\sigma)$ boundedness of $\MM_1$ it follows that
$$\int_{B'} |\MM_1 h|^2\,d\sigma = \int_{B'} |\MM_1 (\chi_{B''}h)|^2\,d\sigma
<\infty,$$
where $B''=\overline B(0,2R+2)$. Also, by Corollary \ref{corovmo},
$$\int_{B'} (K_x)^2\,d\sigma<\infty,$$
and so $\NN_1(\vphi_R\nabla u)\in L^1(\omega^x)$ and \rf{eqclau*20} holds.

From the definition of $h_R$ and \rf{eqclau*20} we deduce that
\begin{equation}\label{eqig31}
\vphi_R(x)\,\nabla u(x) = -\int \vphi_R(y)\,h(y)\,\vec{n}(y)\,d\omega^x(y) + w_R(x).
\end{equation}
Hence, letting $R\to\infty$,
$$|\nabla u(x)| \leq \int |h(y)|\,d\omega^x(y) + \liminf_{R\to\infty} |w_R(x)|.$$
By Lemma \ref{lemtec1} (a) and Corollary \ref{corovmo0} (with $\ve$ small enough), we deduce easily
that $w_R(x)\to 0$ as $R\to\infty$, for any fixed $x\in \Omega$, and then the lemma follows.
\end{proof}

\vv

Now we wish to obtain a variant of Lemma \ref{lemgradgreen1} suitable for the case when the pole for harmonic measure is finite. This is what we do in the next lemma.

\begin{lemma}\label{lemgradgreen2}
Let $\Omega\subset \R^{n+1}$ be a  chord-arc domain. Suppose that the  harmonic measure $\omega^p$ in $\Omega$  with pole  at $p\in\Omega$ satisfies $\log\left(\frac{d\omega^p}{d\sigma}\right)\in \textrm{VMO}(\sigma)$. Then, for all $x\in\Omega$ such 
$d_\Omega(x)\leq d_\Omega(p)/8$ and all $q_x\in\partial\Omega$ such that $|x-q_x|\leq d_\Omega(p)/8$,
\begin{equation}\label{eqgrad2}
|\nabla g(x,p)|\leq \int_{\partial\Omega} K_p(y) \,d\omega^x(y) + C\,\frac{\omega^p\bigl(B(q_x,d_\Omega(p))\bigr)}{d_\Omega(p)^n}\,\left(\frac{d_\Omega(x)}{d_\Omega(p)}\right)^{\alpha/2}.
\end{equation}
\end{lemma}


\begin{proof}
Let $\xi\in\partial\Omega$ and take a $C^\infty$ function $\vphi$ compactly supported in $B(\xi,d_\Omega(p)/4)$ which is
identically $1$ on $B(\xi,d_\Omega(p)/8)$, so that $|\nabla^j\vphi|\lesssim 1/d_\Omega(p)^j$ for $j=1,2$. Note that, in particular, $\vphi$ vanishes on $B(p,d_\Omega(p)/4)$. We consider the function
$$w_0(x) = \int_\Omega g(x,y)\,\Delta[\vphi\,\nabla g(\cdot,p)](y)\,dy\quad \mbox{ for $x\in\Omega$.}$$
We claim that
\begin{equation}\label{eqw01}
|w_0(x)|\lesssim  \frac{\omega(B(\xi,d_\Omega(p)/8))}{d_\Omega(p)^n}\,\left(\frac{d_\Omega(x)}{d_\Omega(p)}\right)^{\alpha/2}\quad\mbox{ if $|x-\xi|\leq \dfrac{d_\Omega(p)}4$.}
\end{equation}

The arguments to prove \rf{eqw01} are quite similar to the ones in Lemma \ref{lemtec1}.
By the  relationship between Green's function and harmonic measure and by standard estimates for positive harmonic functions, for all $y\in B(\xi,d_\Omega(p)/4)\cap \Omega$ we have
$$|\nabla g(y,p)|\lesssim \frac{g(y,p)}{d_\Omega(y)} \approx \frac{\omega^p(B(\xi,d_\Omega(p)/4))}{d_\Omega(y)^n}$$
and
$$|\nabla^2 g(y,p)|\lesssim \frac{g(y,p)}{d_\Omega(y)^2} \approx \frac{\omega^p(B(\xi,d_\Omega(p)/4))}{d_\Omega(y)^{n+1}}
.$$
Thus,
\begin{align*}
|w_0(x)| & = \left|\int_\Omega g(x,y)\,\bigl(\Delta\vphi(y)\,\nabla g (y,p) + 2\nabla \vphi(y)\cdot \nabla^2 g(y,p)\bigr)\,dy\right|\\
& \lesssim \int_{A(\xi,d_\Omega(p)/8,d_\Omega(p)/4)\cap\Omega} g(x,y)\left( 
\frac{\omega^p(B(\xi,d_\Omega(p)/4))}{d_\Omega(p)^2\,d_\Omega(y)^n} + \frac{\omega^p(B(\xi,d_\Omega(p)/4))}{d_\Omega(p)\,d_\Omega(y)^{n+1}}\right)dy\nonumber\\
& \lesssim \int_{B(\xi,d_\Omega(p)/4)\cap\Omega} g(x,y)\,\frac{\omega^p(B(\xi,d_\Omega(p)/4))}{d_\Omega(p)\,d_\Omega(y)^{n+1}}\,dy.\nonumber
\end{align*}
Notice that the integral on the right hand side above is very similar to the one on the right hand side of \rf{eqdjk28}. The reader can check that exactly the same arguments and estimates used to prove Lemma \ref{lemtec1} (a) yield \rf{eqw01}, 
with $\xi$ instead of $0$, $d_\Omega(p)/8$ instead of $R$, $\omega^p$ instead of $\omega$, and $g(y,p)$ instead of $u(y)$.
We leave the details for the reader.

From \rf{eqw01} it follows that $w_0\in C^{\alpha/2}(\overline \Omega)$ and it vanishes at $\partial\Omega$. Further, the function defined by
$$h_0(x)=  \vphi(x)\nabla g(x,p) - w_0(x), \quad\mbox{ $x\in\Omega$},$$
is harmonic in $\Omega$, because $\Delta w_0 = \vphi\,\nabla g(\cdot,p)$. Hence, arguing as in \rf{eqig31}, we derive
$$\vphi(x)\,\nabla g(x,p) = -\int \vphi(y)\,K_p(y)\,\vec{n}(y)\,d\omega^x(y) + w_0(x).$$
If $|x-\xi|\leq d_\Omega(p)/8$, then $\vphi(x) =1$ and from the last identity and the inequality \rf{eqw01}
with $\xi=q_x$, we deduce \rf{eqgrad2}.
\end{proof}

\vvv


\section{The pseudo-blow-up of harmonic measure is surface measure}\label{section6}

Let $\Omega \subset \R^{n+1}$ be a chord-arc domain. We recall that harmonic measure with either a finite pole $p \in \Omega$ or pole at infinity is in the $A_\infty(\sigma)$ class of weights by
\cite{DJ} or \cite{Semmes} and thus, the Poisson kernel $\frac{d \omega}{d\sigma}$ exists and is positive and finite. We denote by $u$ either the Green's function with pole at $p \in \Omega$ or with pole at infinity and $h$ the corresponding Poisson kernel (see \eqref{eq:Green-hminfty} for pole at infinity). 

\subsection{Pseudo-blow-ups of chord-arc domains}\label{sec:pbup}

Here we introduce the notion of {\it pseudo-blow-ups} from \cite{Kenig-Toro2}, but with a slight modification. Let $x_i \in \d \Omega$ and let $\{r_i\}_{i \geq 1}$ be a sequence of positive numbers so that $\lim_{i \to \infty} r_i = 0$. Consider now the domains 
$$ \Omega_i = \frac{1}{r_i} (\Omega -x_i),$$
so that $\d \Omega_i = \frac{1}{r_i} (\d \Omega -x_i)$, 
and the functions $u_i$ in $\Omega_i$ defined by 
$$ u_i (x) = \frac{g(r_i x + x_i,p_{i})}{r_i \,\omega^{p_{i}}(B(x_i,r_i))}\, \sigma(B(x_i, r_i)),$$
where either $p_{i}=\infty$ or $p_i\in \Omega\setminus \{x_i\}$ satisfies
$$\frac{p_{i}-x_i}{r_i}\to \infty \quad \mbox{ as $i\to\infty$.}$$
 Note that $u_i$ vanishes at $\d \Omega_i$ and is harmonic in  $\Omega_i \setminus \{\frac{p_{i}-x_i}{r_i}\}$. We denote by $d \omega_i = h_i \,d\sigma_i$ the harmonic measure of $\Omega_i$ with pole at infinity or $\frac{p_i-x_i}{r_i}$ depending on the pole of $u$, where $\sigma_i=\HH^{n}|_{\d \Omega_i}$. Moreover, the corresponding Poisson kernel\footnote{In fact, this is the Poisson kernel of $\Omega_i$ with pole at $p_i$ modulo a constant factor.} $h_i$ satisfies
$$h_i(x)= \frac{h(r_i x +x_i)}{\omega^{p_{i}}(B(x_i,r_i))}\, \sigma(B(x_i,r_i)).$$

\vv
\begin{theorem}[Theorem 4.1, \cite{Kenig-Toro2}]\label{thm:blowup}
If $\Omega \subset \R^{n+1}$ is a chord-arc domain, then there exists a subsequence satisfying
\begin{align*}
\Omega_i &\to  \Omega_\infty, \quad \textup{in the Hausdorff metric, uniformly on compact sets}\\
\partial \Omega_i &\to \partial \Omega_\infty, \quad \textup{in the Hausdorff metric, uniformly on compact sets},
\end{align*}
where $\Omega_\infty$ is a chord-arc domain. Moreover, there exists $u_\infty \in C(\overline \Omega_\infty)$ such that $u_i \to u_\infty$ uniformly on compact sets which satisfies  \eqref{eq:hminfty} for $\Omega=\Omega_\infty$. Furthermore, $\omega_i \to \omega_\infty$ weakly as Radon measures and $\omega_\infty$ is the harmonic measure of $\Omega_\infty$ with pole at infinity (corresponding to $u_\infty$). 
\end{theorem}

This was originally shown in \cite{Kenig-Toro2} under the assumption that $p_{i}$ is a fixed point and $x_{i}$ converges to some point in $\d\Omega$. However, the same proof gives the result above.

\vv

\begin{theorem}
If $\Omega_\infty \subset \R^{n+1}$ and $u_\infty$ are as in Theorem \ref{thm:blowup}, then
\begin{equation}\label{eq:b-upLip}
\sup_{z \in \Omega_\infty} |\nabla u_\infty(z)| \leq 1.
\end{equation}
\end{theorem}

\vv

\begin{theorem}
If $\Omega_\infty \subset \R^{n+1}$, $u_\infty$ and $\omega_\infty$ are as in Theorem \ref{thm:blowup}, then
\begin{equation}\label{eq:hinftygeq1}
\frac{d\omega_\infty}{d\sigma_\infty} \geq 1, \quad \HH^{n} \textup{-a.e. on}\,\, \d \Omega_\infty,
\end{equation}
where $\sigma_\infty = \HH^{n}|_{\d \Omega_\infty}$.
\end{theorem}

Both theorems were proved in \cite[Theorems 4.2 and 4.3]{Kenig-Toro2} for Reifenberg flat domains with $n$-AD regular boundary, although, an inspection of the proofs shows that the same  arguments, with very minor changes, work also for NTA domains with $n$-AD regular boundary, i.e., for chord-arc domains.

\vv
\begin{coro}\label{cor:hinfty1}
If $\Omega_\infty \subset \R^{n+1}$, $u_\infty$ and $\omega_\infty$ are as in Theorem \ref{thm:blowup}, then
\begin{equation}\label{eq:hinfty1}
|\nabla u_\infty|=\frac{d\omega_\infty}{d\sigma_\infty}=1\quad\mbox{ $\HH^{n}$-a.e. on $\partial\Omega_\infty$.}
\end{equation}
\end{coro}

\begin{proof}
Combining \eqref{eq:n.t.limit} and \eqref{eq:b-upLip} we get that $\frac{d\omega_\infty}{d\sigma_\infty} \leq 1$ for $\HH^{n}$-a.e. on $\partial\Omega_\infty$. Then \eqref{eq:hinfty1} follows from \eqref{eq:hinftygeq1}.
\end{proof}

\vv
\begin{lemma}
The subsequence introduced in Theorem \ref{thm:blowup} satisfies $\sigma_i \rightharpoonup \sigma_\infty$ weakly as Radon measures.
\label{lem:sigma}
\end{lemma}
\begin{proof}
This was essentially proved in Theorem 4.4 in \cite{Kenig-Toro2}. The only difference is that instead of invoking \cite[Theorem 2]{Kenig-Toro2} in the proof, which is particular to the Reifenberg flat case, we just use Corollary \ref{cor:hinfty1}. 
\end{proof}

\vv
\subsection{Blow-downs of unbounded chord-arc domains}\label{sec:bdown}
\def\blowd{\widetilde\Omega}
\def\ublowd{\tilde{u}}
In the course of proving our main result we will need to construct the {\it blow-down} domain with respect to a fixed point $x_0 \in \d \Omega$ of an unbounded chord-arc domain $\Omega$ such that $\frac{d \omega}{d\sigma}=1$  $\sigma$-a.e.\ on $\d \Omega$ (i.e., $\omega=\sigma$). To do so, we let $x_i=x_0$ for all $i\geq 1$ and a sequence of positive numbers $r_i$ such that $\lim_{i \to \infty} r_i= \infty $. Now we take  $\Omega_i$ and $u_i$ as in the construction of pseudo-blow-ups in Subsection \ref{sec:pbup} and 
$p=p_i=\infty$. Then
similar (but easier) arguments show that there exists a chord-arc domain $\blowd$ such that
\begin{align*}
\Omega_i &\to   \blowd, \quad \textup{in the Hausdorff metric, uniformly on compact sets, and}\\
\partial \Omega_i &\to \partial  \blowd, \quad \textup{in the Hausdorff metric, uniformly on compact sets}.
\end{align*}
Moreover, there exists $ \ublowd \in C\bigl(\overline {\blowd}\bigr)$ such that $u_i \to  u_0$ uniformly on compact sets which satisfies 
\begin{equation*}
\left\{ \begin{array}{ll} \Delta  \ublowd = 0 & \mbox{ in $\blowd$,}\\
\ublowd> 0 & \mbox{ in $\blowd$,}\\
 \ublowd=0  & \mbox{ in $\partial \blowd$.}
\end{array}
\right.
\end{equation*}.

\vvv


\section{Application of the monotonicity formula of Weiss: blow-downs are planes in $\R^3$}

We first introduce the notion of variational solution of the one-phase free boundary problem in an open ball $B \subset \R^{n+1}$,

\begin{equation}\label{eq:var.sol}
\left\{ \begin{array}{ll} 
u \geq 0  & \mbox{ in $B$,}\\
\Delta u = 0 & \mbox{ in $B^+(u):=B \cap \{u>0\} $,}\\
|\nabla u| = 1  & \mbox{ on $F(u):= \d B^+(u) \cap B$.}\\
\end{array}
\right.
\end{equation}

\begin{definition}\label{def:vsol}
We define $u \in W^{1,2}_{loc}(B )$ to be a {\it variational solution} of \eqref{eq:var.sol} if 
\begin{enumerate}
\item $u \in C(B ) \cap C^2 \left(B^+(u)\right)$,
\item $\chi_{\{u>0\}} \in L^1_{loc}(B )$ and 
\item the first variation with respect to the functional 
\begin{equation} \label{eq:var.functional}
F(v):= \int_{B}  \left( |\nabla v|^2 + \chi_{\{v>0\}}  \right)\,dm
\end{equation}
vanishes at $v=u$, i.e., 
\begin{align}\label{eq:1st-var}
0 = - \frac{d}{d \epsilon}  F(u(x+\epsilon \phi(x)))|_{\epsilon=0}= \int_{B } \bigl[( |\nabla u|^2 +\chi_{\{u>0\}})  \dv\phi - 2 \nabla u\, D \phi \,(\nabla u)^T\bigr]\,dm, 
\end{align}
for any $\phi \in C^\infty_c(B ; \R^{n+1})$.
\end{enumerate}
\end{definition}

\begin{definition}\label{def:wsol}
We say that $u$ is a {\it weak solution} of $\Delta u = \HH^{n}( \d \{u >0\} \cap \cdot)$ in $B$ if the following are satisfied:
\begin{enumerate}
\item $u \in W^{1,2}_{loc}(B) \cap C(B^+(u))$, $u \geq 0$ in $B$ and $u$ is harmonic in the open set $\{u >0\}$.
\item {\it Non-degeneracy and regularity}: for any open $D \Subset B$ there exist $0<c_D \leq C_D <\infty$ such that for any $B(x,r) \subset D$ satisfying $x \in \d \{ u>0\}$ we have
\begin{equation}\label{eq:wsol-avrg}
c_D \leq r^{-n-1} \int_{\d B(x,r) } u \,d\HH^{n} \leq C_D.
\end{equation}
\item $\{u>0\}$ is locally in $B$  a set of finite perimeter and
\begin{equation}\label{eq:wsol-green}
-\int \nabla u \cdot \nabla \zeta \,dm=  \int_{\d^*\{u>0\}} \zeta \,d\HH^{n}, 
\end{equation}
for any $\zeta \in C^\infty_c(B)$, where $\d^*\{u>0\} $ stands for the reduced boundary of $\{u>0\}$.
\end{enumerate}
\end{definition}

Let us now record a useful lemma whose proof is contained in the one of  \cite[Theorem 5.1]{Weiss}.

\begin{lemma}\label{lem:weakisvar}
If $u$ is a weak solution of $\Delta u =\HH^{n}( \d \{u >0\} \cap \cdot)$ in a ball $B$ in the sense of Definition \ref{def:wsol}, then it is also a variational solution in   the  ball $B$ in the sense of Definition \ref{def:vsol}.
\end{lemma}


\begin{lemma}\label{lem:bup=vsol}
Assume that $\Omega_\infty$ is the blow up domain and $u_\infty$ the blow-up Green's function constructed in Theorem \ref{thm:blowup}. If $B$ is a ball centered on $\d \{u_\infty>0\}=\d \Omega_\infty$, then the extension by zero of $u_\infty$ outside $ \{u_\infty>0\}$ is a weak solution of $\Delta u =\HH^{n}( \d \{u >0\} \cap \cdot)$ in $B$.
 \end{lemma}
 
 \begin{proof}
 By construction, $\Omega_\infty=\{u_\infty >0\}$, $u_\infty> 0$ in $\Omega_\infty$, $u_\infty= 0$ in $\d \Omega_\infty$, $u_\infty$ is harmonic in $\Omega_\infty$, $u_\infty \in C(\overline \Omega_\infty)$ and $|\nabla u_\infty | \leq 1$ in $\Omega_\infty$. Therefore, it is trivial to see that its extension by zero in the complement of $\Omega_\infty$ satisfies the condition (1) in Definition \ref{def:wsol} for the ball $B$. Notice also that by Harnack's inequality at the boundary, if $x_r$ is a corkscrew point in $B(x,r) \cap \Omega_\infty$, it holds 
 $$\max_{z \in \d B(x,r) \cap \Omega_\infty} u_\infty(z)=\max_{z \in B(x,r) \cap \Omega_\infty} u_\infty(z) \approx u_\infty(x_r).$$
 Therefore, we have that by \eqn{wsimu} and Corollary \ref{cor:hinfty1}
 $$ r^{-n-1} \int_{\d B(x,r) } u_\infty \,d\HH^{n} \approx \frac{\HH^{n}(\d B(x,r))}{r^{n+1}} \,u_\infty(x_r) \approx \frac{ \omega_\infty(B(x,r))}{\sigma_\infty(B(x,r))}=1.$$
Since ${\d \Omega_\infty}$ is $n$-AD regular, we have that $\HH^{n}|_{\d \Omega_\infty}$ is locally finite and thus, $\Omega_\infty$ is of locally finite perimeter in $\R^{n+1}$. By the generalized Gauss-Green formula for sets of locally finite perimeter, we infer that
\begin{align*}
\int_{\d \Omega_\infty} \zeta \, d\HH^{n}=\int_{\d \Omega_\infty} \zeta \,d \omega_\infty&=\int_{\Omega_\infty}  u_\infty\,  \Delta \zeta \,dm\\
&= \int_{\Omega_\infty}  \dv (u_\infty  \nabla \zeta) \,dm -  \int_{\Omega_\infty}  \nabla u_\infty \cdot  \nabla \zeta \,dm\\
&= 0-  \int_{\Omega_\infty}  \nabla u_\infty \cdot  \nabla \zeta \,dm, 
\end{align*}
for any $\zeta \in C^\infty_c(\R^{n})$. Note that $\HH^{n}(\d \Omega_\infty \setminus \d^* \Omega_\infty)=0$ in any NTA domain and thus, condition (3) in Definition \ref{def:wsol} is satisfied.
 \end{proof}
 
We state without proof a lemma from \cite{Jerison-Kamburov} which allows us to conclude that any blow-down domain of $\Omega_\infty$ is in fact a cone.

\begin{lemma}{\cite[Lemma 5.2]{Jerison-Kamburov}}\label{lem:blow-down-cone}
Let u be a variational solution of \eqref{eq:var.sol} in $\R^{n+1}$ which is globally Lipschitz. Assume that $0 \in F(u)$ and let $v$ be any limit of a uniformly convergent on compact sets sequence of 
$$v_j(x)=R_j^{-1} u(R_j x),$$
as $R_j \to \infty$. Then $v$ is Lipschitz continuous and homogeneous of degree 1.
\end{lemma}

\vv

\begin{lemma}\label{lemweiss}
Assume that $\Omega_\infty\subset\R^{n+1}$ is the blow-up domain and $u_\infty$ the blow-up Green's function constructed in Theorem \ref{thm:blowup}. If $ x \in\d \Omega_\infty$, then any blow-down domain of $\Omega_\infty$ at $x$ is a cone.
\end{lemma}

By a cone we mean a set $F\subset \R^{n+1}$ such that if $x\in F$, then $\lambda x\in F$
for all $\lambda>0$. A conical domain is a domain which is a cone.

\begin{proof}
It follows from Lemmas \ref{lem:weakisvar}, \ref{lem:bup=vsol} and \ref{lem:blow-down-cone} in view of Subsection \ref{sec:bdown}.
\end{proof}

\vv
\begin{lemma}\label{lemhalf}
If $\Omega_0\subset\R^3$ is a conical two-sided NTA domain in $\R^{3}$ with $2$-AD-regular boundary such that 
$\frac{d\omega_0}{d\sigma_0}=1$ $\sigma_0$-a.e. in $\partial\Omega_0$, then $\Omega_0$ is a half-space.
\end{lemma}

\begin{proof}
Since $\Omega_0$ is a conical two-sided NTA domain, the intersection of $\Omega_0$ with the sphere $S^2$ is an open
connected subset of $S^2$, and the interior of its complement should be another open connected set of $S^2$. Further, as shown in
\cite{CJK} (see Remark 2 and p.\ 92 in this reference) by studying the mean curvature of $\partial\Omega_0\cap S^2$, one deduces that $\partial\Omega_0\cap S^2$ is a convex curve and
$\Omega_0^c$ is a convex cone. One can check that a convex cone in $\R^3$ is a Lipschitz domain, and also its exterior domain.
Hence, by the results of Farina and Valdinoci \cite{Farina-Valdinoci} (or by arguments analogous to the ones in \cite[p.\ 92]{CJK}), $\Omega_0$ is a half-space.
\end{proof}

\vv


\begin{coro}\label{corokey1}
Suppose that $\Omega_0$ is a two-sided NTA domain in $\R^3$ with $2$-AD-regular boundary such that 
$\frac{d\omega_0}{d\sigma_0}=1$ $\sigma_0$-a.e.\ in $\partial\Omega_0$.
Then, for any $x\in\partial \Omega_0$,
$$\lim_{r\to\infty}\Theta_{\d\Omega_0}(x,r)=0.$$
\end{coro}

\begin{proof}
This is an immediate consequence of Lemmas \ref{lemweiss} and \ref{lemhalf}.
\end{proof}

\vvv


\section{The Alt-Caffarelli theorem}

The objective of this section is to explain how to prove the following lemma.

\begin{lemma}\label{lemac23}
Let $\Omega_0$ be an NTA domain in $\R^{n+1}$ with $n$-AD-regular boundary with constant $C_0$.  Suppose $0\in \d\Omega_{0}$ and 
\begin{equation}\label{e:w=1}
\frac{d\omega_0}{d\sigma_0}\equiv1\;\;\; \sigma_0\mbox{-a.e. in }\;\; \partial\Omega_0.
\end{equation}
There exists $\delta_{0}>0$ small enough depending on  $n$, the NTA character of $\Omega_0$, and $C_0$ such that if $\bB=B(0,1)$ satisfies
\begin{equation}\label{e:bt<d}
\Theta_{\d\Omega_0}(\lambda \bB)\leq \delta_0
\quad \mbox{for all $\lambda>1$,}
\end{equation}
then $\Omega_{0}$ is a half-space. 
\end{lemma} 
\vv

Before turning to the proof of this lemma, notice that an
 immediate consequence of this and Corollary \ref{corokey1} is the following.

\begin{coro}\label{corokey**}
Suppose that $\Omega_0$ is a two-sided chord-arc in $\R^3$ such that 
$\frac{d\omega_0}{d\sigma_0}=1$ $\sigma_0$-a.e.\ in $\partial\Omega_0$.
Then, $\Omega_0$ is a half space.
\end{coro}

Lemma \ref{lemac23} is essentially proven in \cite{Kenig-Toro3}, which assumes that the domain is Reifenberg flat. This is a variant of some of the results by Alt and Caffarelli in \cite{Alt-Caffarelli}. In \cite{Kenig-Toro3}
the authors also assume in the statement of their theorem that
\begin{equation}\label{e:u<1}
|\nabla u_{0}|\leq \chi_{\Omega}
\end{equation}
where $u_{0}$ is its Green function with pole at infinity. However, this estimate is an immediate consequence of the assumptions of Lemma \ref{lemac23} (specially, \eqn{w=1}) and Lemma \ref{lemgradgreen1}. Thus, we will only explain how to read and adjust the proof in \cite{Kenig-Toro3} in order to obtain the lemma, adding details where necessary.

\begin{lemma}\label{lem8.3}
Let $\Omega\subset \R^{n+1}$ be a two-sided $C$-corkscrew domain so that $\Omega_{\ext}$ is also connected. Then whenever $\xi\in \d\Omega$, $r>0$, and $\beta_{\d\Omega}(\xi,r,P)<\frac{1}{2C}$ for some $n$-plane $P$, 
\begin{equation}\label{e:t<b}
\Theta_{\d\Omega}(\xi,r/2,P)\leq 2 \,\beta_{\d\Omega}(\xi,r,P)
\end{equation}
and there are half spaces $H^{\pm}$ such that 
\[
H^{+}\cup H^{-}=\bigl\{y:\dist(y,P)> \beta_{\d\Omega}(\xi,r,P) \bigr\},\] 
\[
H^{+}\cap B(\xi,r)\subset \Omega, \mbox{ and }  H^{-}\cap B(\xi,r)\subset \Omega_{\ext}.\] 
In particular, if $\pi_{P}$ is the projection onto $P$, then $\pi_{P}(\d\Omega\cap B(\xi,r))\supseteq \pi_P(B(\xi,r/2))$.
\end{lemma}

\begin{proof}
Without loss of generality, we assume $\xi=0$, $r=1$, so $B(\xi,r)=\bB=B(0,1)$. Let $\ve=\beta_{\d\Omega}(\xi,r,P)$. If $(H^{+}\cup H^{-})\cap \bB\subset \Omega$, then 
\[\Omega_{\ext}\cap \bB\subset \{y: \dist(y,P)\leq \ve\},\] 
but since $\Omega$ has exterior corkscrews, there must be 
\[
B(y,1/C)\subset \bB \cap \Omega_{\ext}\subset \{y: \dist(y,P)\leq \ve\},\]
which is a contradiction for $\ve<\frac{1}{2C}$. We also get a contradiction if $(H^{+}\cup H^{-})\cap \bB\subset \Omega_{\ext}$, and so $H^{\pm}\cap \bB$ must be in two different components. Assume $H^{+}\cap \bB \subset \Omega$ and $H^{-}\subset \Omega_{\ext}$. The last part of the lemma now follows from this, since for any $y\in \pi_{P}(B(\xi,r))$, the line $\pi_{P}^{-1}(y)$ must pass through both $H^{\pm}$, and thus it must intersect $\d\Omega$. 

To prove \eqn{t<b} it suffices to show that if $x\in \frac{1}{2}\bB\cap P$, then $\dist(x,\d\Omega)\leq 2\ve$. Suppose there is { $x\in \frac{1}{2}\bB\cap P$} so that $B(x,2\ve)\subset (\d\Omega)^{c}$. Then the set
\[
U=\bB\cap B(x,2\ve )^{c} \cap \{y:\dist(y,P)> \ve\}\]
is a connected open subset of $(\d\Omega)^{c}$, and hence $U\subset \Omega$ or $U\subset \Omega_{\ext}$. Without loss of generality, we can assume the former case. Then 
\[\Omega_{\ext}\cap \bB\subset \{y:\dist(y,P)\leq \ve\}\cup B(x,2\ve).\] 
But by the exterior corkscrew condition, there is $B(y,1/C)\subset \bB\cap \Omega_{\ext}$, which is impossible if $\ve<\frac{1}{2C}$. 
\end{proof}


The following definition comes from \cite{Kenig-Toro3}, and it is a variant of one that appears in \cite{Alt-Caffarelli}. 
\begin{definition}
Let $\Omega\subset \R^{n+1}$ be an NTA domain. Let $x_{0}\in \d\Omega$, $\rho>0$, $\sigma_{+},\sigma\in (0,1)$, $\nu\in \bS^{n}$, and $v$ be the Green function with pole at infinity. We say $v\in F(\sigma_{+},\sigma)$ in $B(x_0,\rho)$ in the direction $\nu\in \bS^{n}$ if, for all $x\in B(x_{0},\rho)$,
\begin{equation}\label{e:fss1}
v(x)=0 \quad \mbox{ if }\;(x-x_{0})\cdot \nu \geq \sigma_{+}\rho
\end{equation}
and
\begin{equation}\label{e:fss2}
v(x)\geq -(x-x_{0})\cdot \nu -\sigma \rho\quad \mbox{ if }\;(x-x_{0})\cdot \nu \leq -\sigma \rho.
\end{equation}
\end{definition}

Observe that $v\equiv 0$ exactly on $\Omega^{c}$ and $v>0$ exactly on $\Omega$, and so 
\begin{equation}\label{e:sigmabeta}
\mbox{$v\in F(\sigma,\sigma)$ in direction $\nu$ in $B(x_{0},\rho)$ implies $\beta_{\d\Omega}(x_{0},\rho)\leq \sigma$.}
\end{equation} Indeed, assume $x_{0}=0$, $\rho=1$, and note that  by \eqn{fss1}, since $v=0$ only when $\Omega^{c}$, we have that for $x\in B(x_{0},\rho)$ that 
\[
\{x\in \bB: x\cdot \nu \geq \sigma \}\subseteq \Omega^{c}.\]
By \eqn{fss2}, if $x\cdot \nu<-\sigma \rho$, then
\[
v(x)\geq -x\cdot \nu-\sigma >0\]
and
 since $v(x)>0$ only when $x\in \Omega$
\[
\{x\in \bB: x\cdot \nu < -\sigma \}\subseteq \Omega.\]
Since $v$ is continuous, we thus have 
\[
\beta_{\d\Omega}(0, 1)<\sigma.\]

\begin{lemma}\label{l:ktlemma1}
Let $\Omega$ be a two-sided NTA domain and $v$ the Green function with pole at infinity. 
Let $x_{0}\in \partial\Omega$, $\rho,\sigma>0$, and $\nu\in \mathbb{S}^{n}$. If $v\in F(\sigma,1)$ in $B(x_{0},\rho)$ in the direction $\nu$, then $v\in F(2\sigma,C\sigma)$ in $B(x_0 ,\rho/2)$ in the same direction, where $C=C(n)$.
\end{lemma}

\begin{proof}
The proof is exactly the same as in Lemma 0.4 in \cite{Kenig-Toro3}. Its proof and that of Lemma 0.3 in the same paper upon which it depends do not require the Reifenberg flat assumption and the proofs are identical.
\end{proof}

\begin{lemma}\label{l:ktlemma2}
Let $\Omega$ be a two-sided NTA domain and $v$ the Green function with pole at infinity. 
There is some $\ve_0$ small enough so that the following holds.
Let $x_0 \in \partial\Omega$, $\rho>0$ and $\nu\in \mathbb{S}^{n}$. Given $\theta\in (0,1)$, there is $\sigma_{\theta}>0$ and $\eta\in (0,1)$ so that if $0<\sigma<\sigma_{\theta}$ and $v\in F(\sigma,\sigma)$ in $B(x_0 ,\rho)$ in the direction $\nu$ and $\beta_{\d\Omega}(x_0 ,2\rho)<\ve_0$, then $v\in F(\theta\sigma,1)$ in $B(x_0 ,\eta \rho)$ in some direction $\nu'$ such that $|\nu-\nu'|<C\sigma$. 
\end{lemma}

\begin{proof}
Again, the proof is exactly the same as that of Lemma 0.5 in \cite{Kenig-Toro3}. The only time Kenig and Toro use the Reifenberg flatness assumption is to show that the intersection of a cylinder C with the boundary (with axis passing through $Q_0$) has projection in the direction of the cylinder equal to the base of the cylinder (i.e. a ball), see right below equation (0.69) in \cite{Kenig-Toro3}. However, we can just replace this with the assumption that $\beta_{\d\Omega}(x_0 ,2\rho)<\ve_0$ is small and then apply Lemma \ref{lem8.3}. 
\end{proof}
 
 \begin{proof}[Proof of Lemma \ref{lemac23}]
Let $\theta'\in (0,1/2)$, $\delta_{0}\in (0,\sigma_{n,\theta'}/(8+2C))$. Note that \eqn{bt<d} implies that for $r>1$, there is a plane $P_{r}$ so that 
\begin{equation}\label{e:assumetheta}
\beta_{\d\Omega}(0,r,P_{r})\leq \Theta_{\d\Omega}(0,r,P_{r})\leq \delta_{0}.
\end{equation} Let $L_{r}=P_{2r}-\pi_{P_{2r}}(0)$ and let $\nu_{r}\in \bS^{n}$ be a unit vector orthogonal to $L_{r}$ so that $r\nu_{r}/2 \in \Omega^{c}$. Then 
\[
\{x\in B(0,r): x\cdot \nu_{r}> \delta_{0} r\}
\subseteq \{x\in B(0,r): \dist(x,L_{r})>\delta_{0}r\} \subseteq (\d\Omega)^{c}.\]
Since $\{x\in B(0,r): x\cdot \nu_{r}> \delta_{0} r\}$ and $\Omega^{c}$ are connected and $r\nu_{r}/2 $ is in their intersection, we actually have
\[
\{x\in B(0,r): x\cdot \nu_{r}> \delta_{0} r\}\subseteq \Omega^{c}.\]
Hence, $v(x)=0$ for $x\in B(0,r)$ such that $x\cdot \nu>\delta_{0} r$. Furthermore, we trivially have
\[
\{x\in B(0,r): x\cdot \nu_{r}>r\}=\varnothing\]
and thus $v\in F(\delta_{0},1)$. \Lemma{ktlemma1} implies $v\in F(2\delta_{0},C\delta_{0})$ in $\frac{r}{2}\bB$ in the same direction, and so $v\in F(\delta,\delta)$  in $\frac{r}{2}\bB$ where $\delta=\max\{2,C\}$. Let $\theta'\in (0,1)$. By \Lemma{ktlemma2} and \eqn{assumetheta}, there is $\eta'\in (0,1)$ (depending only on $\theta'$) so that $v\in F(\theta' \delta,1)$ in $\frac{\eta' r}{2}\bB$. Again, by \Lemma{ktlemma1}, we have $v\in F(2\theta'\delta,C\theta'\delta)$ in $\frac{\eta' r}{4}\bB$, and hence $v\in F(\theta \delta,\theta\delta)$ in $\eta r\bB$ where $\theta=\max\{2\theta',C\theta'\}$ and $\eta=\frac{\eta'}{4}$ in the direction of some vector $\nu\in\bS^{n}$. By \eqn{sigmabeta}, we have
\[
\beta_{\d\Omega}(0,\eta r)<\theta \delta.\]

 
Iterating, we get that for all $m\in \mathbb{N}$ that

\begin{equation}\label{e:iterate1}
v\in F(\theta^{m}\delta,\theta^{m}\delta)\quad\, \mbox{in }\;\eta^{m}r\bB.
\end{equation}
and
\[
\Theta_{\d\Omega}(0,\eta^{m}r/2)\leq 2\beta_{\d\Omega}(0,\eta^{m}r) \leq \theta^{m} \delta.\]
Let $1<s\ll r$ and pick $m$ so that $\eta^{m+1}r\leq s<\eta^{m}r$. Then this implies
\[\Theta_{\d\Omega}(0,s/2)
\leq 2\Theta_{\d\Omega}(0,\eta^{m}r/2) 
\leq 2\theta^{m} \delta
=2\eta^{\frac{\log \theta}{\log \eta} m} \delta 
\leq 2(\eta^{-1}sr^{-1})^{\frac{\log \theta}{\log \eta} }\delta.\]
Thus, by sending $r\rightarrow \infty$, we get $\Theta_{\d\Omega}(0,s/2)=0$. Since this holds for every $s>1$, we have that $\d\Omega$ is equal to an $n$-plane, and since $\Omega$ is connected, it must be a halfspace. 
 \end{proof}

\vvv


\section{The proof of Theorem \ref{teo}}

Our arguments are very similar to the ones in \cite{Kenig-Toro2}. The only difference is that in our pseudo-blow-ups we allow the
points $x_i$ to escape to $\infty$. In this way, we are able to show that the outer unit normal $\vec n$ belongs to
$\textrm{VMO}(\sigma)$, not only to $\VMO_{loc}(\sigma)$. For the reader's convenience, we replicate the arguments of  \cite{Kenig-Toro2}
here.

Let 
\[
\ell=\lim_{r\rightarrow 0} \sup_{x\in \d\Omega} \|\vec{n}\|_{*}(B(x,r)).\]
We will show $\ell=0$. Let $x_{i}\in \d\Omega$ and $r_{i}\downarrow 0$ be such that 
\[
\lim_{i\rightarrow \infty}\left(\avint_{B(x_{i},r_{i})} |\vec{n}-\vec{n}_{B(x_{i},r_{i})}|^{2}\,d\sigma\right)^{\frac{1}{2}}=\ell.\]
Let $\Omega_{i}=\frac{1}{r_{i}}(\Omega-x_{i})$ and $u_{i}^{p_{i}},\omega_{i}^{p_{i}}$ be as in Theorem \ref{thm:blowup}. By this theorem, we can pass to a subsequence so that all these quantities converge to some $\Omega_{\infty}$, $u_{\infty}$, and $\omega_{\infty}$. By Lemma \ref{lem:sigma}, $\sigma_{i}$ also converges to $\sigma_{\infty}=\HH^{n}|_{\d\Omega_{\infty}}$. By Lemma \ref{lemhalf}, $\Omega_{\infty}$ is a half space (suppose it is $\mathbb{R}^{n+1}_{+}$) and $\omega_{\infty}=\HH^{n}|_{\mathbb{R}^{n}}$. For $\phi$ a smooth, nonnegative, and compactly supported function with $\phi\geq \chi_\bB$, and $\vec{n}_{i}$ the outer unit normal to $\d\Omega_{i}$, we thus have
\begin{align*}
\lim_{i\rightarrow\infty}   \int_{\d\Omega_{i}\cap \bB}|\vec{n}_{i}  + e_{n+1}|^{2}\,d\sigma_{i}
&\leq \lim_{i\rightarrow\infty} \int_{\d\Omega_{i}} \phi \,|\vec{n}_{i}+e_{n+1}|^{2}\,d\sigma_{i}\\
& =\lim_{i\rightarrow\infty}\left( { 2}  \int_{\d\Omega_{i}} \phi\, d\sigma_{i}+ 2\int_{\d\Omega_{i}} \phi\, \vec{n}_{i}\cdot e_{n+1}\,d\sigma_{i}\right)\\
& 
 = { 2} \int_{\mathbb{R}^{n}} \phi\, d\sigma_\infty + 2\lim_{i\rightarrow\infty}\int_{\Omega_{i}} \textrm{div}(\phi\, e_{n+1})\,dm \\
& = { 2}  \int_{\mathbb{R}^{n}} \phi\, d\sigma_\infty+ 2\int_{\mathbb{R}^{n+1}_{+}} \textrm{div}(\phi \,e_{n+1})\, dm\\
& ={ 2}  \int_{\mathbb{R}^{n}} \phi \,d\sigma_\infty- { 2}  \int_{\mathbb{R}^{n}} \phi\,  e_{n+1}\cdot e_{n+1}\, d\sigma_\infty
=0
\end{align*}
and hence
\[
\ell 
=\lim_{i\rightarrow \infty}\left(\avint_{B(x_{i},r_{i})} |\vec{n}-\vec{n}_{B(x_{i},r_{i})}|^{2}d\sigma\right)^{\frac{1}{2}}
\leq 2 \lim_{i\rightarrow \infty}\left(\avint_{B(x_{i},r_{i})} |\vec{n}+e_{n+1}|^{2}d\sigma\right)^{\frac{1}{2}}=0.\]
\fiproof

\vv

\begin{rem}\label{rem***}
The same arguments as above  show that Theorem A by Kenig and Toro is valid as stated in the Introduction.
That is, under the assumptions of Theorem A, one deduces that $\vec n\in \textrm{VMO}(\sigma)$, instead of the weaker statement $\vec n\in {\rm VMO}_{loc}(\sigma)$ proven in \cite{Kenig-Toro2}.
\end{rem}

\vvv

\section{Counterexample for $\R^d$, $d \geq 4$}\label{sec:counter}

In this section we show that, for all $d\geq4$, there exists a two-sided chord-arc unbounded domain $\Omega \subset \R^d$ for which the Poisson kernel with pole at infinity is constant and such that the outer unit normal is not in $\textrm{VMO}(\sigma)$.
Indeed, in \cite[Example 1]{Hong}, Hong constructed $u \in C(\R^4)$ such that $u \geq 0$;  $u(rx)=r u(x), r>0$; $\Delta u=0$ in $\Gamma = \{u >0\}$; $\d \Gamma \setminus \{0\}$ is smooth; $\frac{\d u}{\d \vec n} = -1$, where $\vec n$ is the outward unit normal on $\Gamma$ and
$u$ is singular, i.e., $u \not = x_1^+$ (modulo rotations). We describe his example in some detail below.
 
Since $u$ is homogeneous of degree one, it is determined by its values on the unit sphere $\mathbb S^3 \subset \R^4$. Further,
$u$ solves the following overdetermined first eigenvalue problem on $\mathbb S^{d-1}$, for $d=4$:
\begin{equation}\label{eq:overdet}
\begin{cases} 
\Delta_{\mathbb S^{d-1}} u +(d-1) u =0 \quad \textup{and} \quad u>0, & \textup{in}\,\, \Omega:= \Gamma \cap \mathbb{S}^{d-1}; \\ 
\dfrac{\d u}{\d \vec n}=-1 \quad  \textup{and} \quad  u=0, & \textup{in}\,\, \d \Omega:= \d \Gamma \cap \mathbb{S}^{d-1};\\
u \equiv 0, &\textup{in}\,\, \Omega^c.
\end{cases}
\end{equation}

To be more precise, let us consider in $\mathbb S^3 \subset \R^4$ the coordinates
\begin{equation}\label{eq:coord}
\begin{cases} 
x_1=\cos \theta \cos \phi, &x_2= \cos \theta \sin \phi \\ 
x_3=\sin \theta \cos \psi, &x_4= \sin \theta \sin \psi,
\end{cases}
\end{equation}
where $\theta \in [0, \pi/2]$ and $\phi, \psi \in [0, 2\pi]$. Let $u(\theta, \phi, \psi)= \tau f(\theta)$, where $\tau>0$ and  $f$ a sufficiently nice function. To find $u$ that satisfies \eqref{eq:overdet} it is enough to solve the following ODE:
\[
\begin{cases} 
(\sin \theta \cos \theta f')'+ \sin \theta \cos \theta f =0, \quad \theta \in (0, \pi/2) \\ 
f(0)=1, \quad f'(0)=0.
\end{cases}
\] 
Then it is shown in \cite{Hong} that there exists $\theta_0 \in (0, \pi/2)$ such that $f(\theta_0)=0$,  $f'(\theta_0)<0$ and $f'(\theta)>0$ for all $\theta \in (0, \theta_0)$. If $u$ is defined on $\mathbb S^3$ by $u(\theta, \phi, \psi) = \frac{-1}{f'(\theta_0)} f(\theta)$, for all $\theta \in [0, \theta_0)$ and $u \equiv 0$ in $[\theta_0, \pi/2]$, then $v(x)= v(r \xi)= r u(\xi)$, for $r>0$ and $\xi \in \mathbb S^3$, is the solution to the one-phase free boundary problem we are after.

The above mentioned construction provides us with a domain for which  Theorem \ref{teo} does not hold. Indeed, let 
\begin{align*}
\Omega:&=\{ x \in \R^4: x = r \xi, \,\, \textup{for some}\,\, \xi \in \mathbb S^3 \,\,\textup{satisfying}\,\, \eqref{eq:coord} \,\, \textup{for}\,\, \theta \in [0, \theta_0)\}\\
&= \{v>0\},
\end{align*} 
whose boundary is given by all points $x \in \R^4$ so that $x = r\, \xi$ for some $r>0$ and $\xi \in \mathbb{S}^3$ that satisfies \eqref{eq:coord} for $\theta=\theta_0$. Remark here that as $v$ is homogeneous of degree one function and $\not \equiv x_1^+$ (under rotation), $\Omega$ is a cone in $\R^4$ but not a half-space. Thus, $\Omega$ is not a Reifenberg flat domain with vanishing constant, which infers that the outward unit normal $\vec n \not \in \textrm{VMO}(\d \Omega)$. Moreover, as the Poisson kernel $h=-\frac{\d u}{\d \vec n}=1$, it is clear that $\log h \in \textrm{VMO}$. Therefore, it is enough to show that $\Omega$ is a 2-sided chord-arc domain.

To this end, notice that every $x \in \d \Omega$ satisfies the equation $x_1^2+x_2^2 = \cos^2\theta_0$
$x_3^2+x_4^2 = \sin^2\theta_0$ while, for $x\in\Omega$,
$$
x_1^2+x_2^2 = \cos^2\theta > \cos^2 \theta_0 \quad \mbox{ and } \quad x_3^2+x_4^2 = \sin^2\theta < \sin^2 \theta_0. 
$$
So $\Omega$ coincides with the set of those points $x\in\R^4$ such that
\begin{equation}\label{eq:cntrxOM}
x_1^2+x_2^2 >\bigl(x_3^2+x_4^2\bigr) \cot^2 \theta_0.
\end{equation}
Therefore, $\Omega$ is bi-lipschitz equivalent to the domain $\{x \in \R^4: x_1^2+x_2^2 >x_3^2+x_4^2\}$, which is a well-known 2-sided chord-arc domain. The AD-regularity is easier to see as the boundary is locally a Lipschitz graph away from the origin by the implicit function theorem, so it is locally 
AD-regular, and the fact that it is a cone easily gives that it is globally Ahlfors regular. Hence, $\Omega$ is also a two-sided chord-arc domain, which finishes our proof in $\R^4$. 
\vv

If we set $D:= \Omega \otimes  \R^{d-4} \subset \R^{d}$, where $\Omega \subset \R^4$ is the domain just constructed, then $D$ is a 2-sided chord-arc domain in $\R^{d}$ for which the Poisson kernel is constant and such that the outer unit normal is not in $\textrm{VMO}(\sigma)$.

\end{document}